\documentclass[11pt,twoside,a4paper]{article}
\usepackage{amsmath}
\usepackage{amssymb}
\usepackage{color}
\usepackage[latin1]{inputenc}
\usepackage[T1]{fontenc}   
\usepackage[english]{babel}
\newcommand{\tun}{\begin{picture}(5,0)(-2,-1)
\put(0,0){\circle*{2}}
\end{picture}}

\newcommand{\tdeux}{\begin{picture}(7,7)(0,-1)
\put(3,0){\circle*{2}}
\put(3,5){\circle*{2}}
\put(3,0){\line(0,1){5}}
\end{picture}}

\newcommand{\ttroisun}{\begin{picture}(15,12)(-5,-1)
\put(3,0){\circle*{2}}
\put(6,7){\circle*{2}}
\put(0,7){\circle*{2}}
\put(-0.65,0){$\vee$}
\end{picture}}

\newcommand{\ttroisdeux}{\begin{picture}(5,15)(-2,-1)
\put(0,0){\circle*{2}}
\put(0,5){\circle*{2}}
\put(0,10){\circle*{2}}
\put(0,0){\line(0,1){5}}
\put(0,5){\line(0,1){5}}
\end{picture}}

\newcommand{\tquatreun}{\begin{picture}(15,12)(-5,-1)
\put(3,0){\circle*{2}}
\put(6,7){\circle*{2}}
\put(0,7){\circle*{2}}
\put(3,7){\circle*{2}}
\put(-0.65,0){$\vee$}
\put(3,0){\line(0,1){7}}
\end{picture}}

\newcommand{\tquatredeux}{\begin{picture}(15,18)(-5,-1)
\put(3,0){\circle*{2}}
\put(6,7){\circle*{2}}
\put(0,7){\circle*{2}}
\put(0,14){\circle*{2}}
\put(-0.65,0){$\vee$}
\put(0,7){\line(0,1){7}}
\end{picture}}

\newcommand{\tquatretrois}{\begin{picture}(15,18)(-5,-1)
\put(3,0){\circle*{2}}
\put(6,7){\circle*{2}}
\put(0,7){\circle*{2}}
\put(6,14){\circle*{2}}
\put(-0.65,0){$\vee$}
\put(6,7){\line(0,1){7}}
\end{picture}}

\newcommand{\tquatrequatre}{\begin{picture}(15,18)(-5,-1)
\put(3,5){\circle*{2}}
\put(6,12){\circle*{2}}
\put(0,12){\circle*{2}}
\put(3,0){\circle*{2}}
\put(-0.65,5){$\vee$}
\put(3,0){\line(0,1){5}}
\end{picture}}

\newcommand{\tquatrecinq}{\begin{picture}(9,19)(-2,-1)
\put(0,0){\circle*{2}}
\put(0,5){\circle*{2}}
\put(0,10){\circle*{2}}
\put(0,15){\circle*{2}}
\put(0,0){\line(0,1){5}}
\put(0,5){\line(0,1){5}}
\put(0,10){\line(0,1){5}}
\end{picture}}

\newcommand{\tdun}[1]{\begin{picture}(10,5)(-2,-1)
\put(0,0){\circle*{2}}
\put(3,-2){\tiny #1}
\end{picture}}

\newcommand{\tddeux}[2]{\begin{picture}(12,5)(0,-1)
\put(3,0){\circle*{2}}
\put(3,5){\circle*{2}}
\put(3,0){\line(0,1){5}}
\put(6,-2){\tiny #1}
\put(6,3){\tiny #2}
\end{picture}}

\newcommand{\tdtroisun}[3]{\begin{picture}(20,12)(-5,-1)
\put(3,0){\circle*{2}}
\put(6,7){\circle*{2}}
\put(0,7){\circle*{2}}
\put(-0.65,0){$\vee$}
\put(5,-2){\tiny #1}
\put(9,5){\tiny #2}
\put(-5,5){\tiny #3}
\end{picture}}

\newcommand{\tdtroisdeux}[3]{\begin{picture}(12,15)(-2,-1)
\put(0,0){\circle*{2}}
\put(0,5){\circle*{2}}
\put(0,10){\circle*{2}}
\put(0,0){\line(0,1){5}}
\put(0,5){\line(0,1){5}}
\put(3,-2){\tiny #1}
\put(3,3){\tiny #2}
\put(3,9){\tiny #3}
\end{picture}}

\newcommand{\tdquatretrois}[4]{\begin{picture}(20,20)(-5,-1)
\put(3,0){\circle*{2}}
\put(6,7){\circle*{2}}
\put(0,7){\circle*{2}}
\put(6,14){\circle*{2}}
\put(-.65,0){$\vee$}
\put(6,7){\line(0,1){7}}
\put(5,-2){\tiny #1}
\put(9,5){\tiny #2}
\put(-5,5){\tiny #4}
\put(9,12){\tiny #3}
\end{picture}}

\newcommand{\tdquatrecinq}[4]{\begin{picture}(12,19)(-2,-1)
\put(0,0){\circle*{2}}
\put(0,5){\circle*{2}}
\put(0,10){\circle*{2}}
\put(0,15){\circle*{2}}
\put(0,0){\line(0,1){5}}
\put(0,5){\line(0,1){5}}
\put(0,10){\line(0,1){5}}
\put(3,-2){\tiny #1}
\put(3,3){\tiny #2}
\put(3,9){\tiny #3}
\put(3,14){\tiny #4}
\end{picture}}

\newcommand{\hdeux}{\begin{picture}(12,8)(-3,-1)
\textcolor{red}{\put(0,0){\circle*{2}}
\put(5,0){\circle*{2}}
\put(0,0){\line(1,0){5}}}
\end{picture}}

\newcommand{\htroisun}{\begin{picture}(12,8)(-3,-1)
\put(3,0){\circle*{2}}
\put(-0.6,0.1){$\vee$}
\textcolor{red}{\put(6,7){\circle*{2}}
\put(0,7){\circle*{2}}
\put(0,7){\line(1,0){5}}}
\end{picture}}

\newcommand{\htroisdeux}{\begin{picture}(12,8)(-3,-1)
\put(0,5){\circle*{2}}
\put(0,0){\line(0,1){5}}
\textcolor{red}{\put(0,0){\circle*{2}}
\put(5,0){\circle*{2}}
\put(0,0){\line(1,0){5}}}
\end{picture}}

\newcommand{\htroistrois}{\begin{picture}(12,8)(-3,-1)
\put(5,5){\circle*{2}}
\put(5,0){\line(0,1){5}}
\textcolor{red}{\put(5,0){\circle*{2}}
\put(0,0){\circle*{2}}
\put(0,0){\line(1,0){5}}}
\end{picture}}

\newcommand{\htroisquatre}{\begin{picture}(17,0)(-3,-1)
\textcolor{red}{\put(0,0){\circle*{2}}
\put(5,0){\circle*{2}}
\put(10,0){\circle*{2}}
\put(0,0){\line(1,0){10}}}
\end{picture}}

\newcommand{\hquatreun}{\begin{picture}(12,12)(-3,-1)
\put(3,0){\circle*{2}}
\put(6,7){\circle*{2}}
\put(-0.5,0.1){$\vee$}
\put(3,0){\line(0,1){7}}
\textcolor{red}{\put(0,7){\line(1,0){2}}
\put(0,7){\circle*{2}}
\put(3,7){\circle*{2}}}
\end{picture}}

\newcommand{\hquatredeux}{\begin{picture}(12,12)(-3,-1)
\put(3,0){\circle*{2}}
\put(0,7){\circle*{2}}
\put(-0.5,0.1){$\vee$}
\put(3,0){\line(0,1){7}}
\textcolor{red}{\put(3,7){\line(1,0){2}}
\put(6,7){\circle*{2}}
\put(3,7){\circle*{2}}}
\end{picture}}

\newcommand{\hquatretrois}{\begin{picture}(12,12)(-3,-1)
\put(3,0){\circle*{2}}
\put(-0.5,0.1){$\vee$}
\put(3,0){\line(0,1){7}}
\textcolor{red}{\put(0,7){\line(1,0){5}}
\put(6,7){\circle*{2}}
\put(0,7){\circle*{2}}
\put(3,7){\circle*{2}}}
\end{picture}}

\newcommand{\hquatrequatre}{\begin{picture}(12,18)(-3,-1)
\put(3,0){\circle*{2}}
\put(0,14){\circle*{2}}
\put(-0.5,0.1){$\vee$}
\put(0,7){\line(0,1){7}}
\textcolor{red}{\put(0,7){\line(1,0){5}}
\put(6,7){\circle*{2}}
\put(0,7){\circle*{2}}}
\end{picture}}

\newcommand{\hquatrecinq}{\begin{picture}(12,18)(-3,-1)
\put(3,0){\circle*{2}}
\put(6,14){\circle*{2}}
\put(-0.5,0.1){$\vee$}
\put(6,7){\line(0,1){7}}
\textcolor{red}{\put(0,7){\line(1,0){5}}
\put(6,7){\circle*{2}}
\put(0,7){\circle*{2}}}
\end{picture}}

\newcommand{\hquatresix}{\begin{picture}(12,18)(-3,-1)
\put(3,5){\circle*{2}}
\put(3,0){\circle*{2}}
\put(-0.5,5.1){$\vee$}
\put(3,0){\line(0,1){5}}
\textcolor{red}{\put(0,12){\line(1,0){5}}
\put(6,12){\circle*{2}}
\put(0,12){\circle*{2}}}
\end{picture}}

\newcommand{\hquatresept}{\begin{picture}(18,8)(-5,-1)
\put(6,7){\circle*{2}}
\put(0,7){\circle*{2}}
\put(-0.6,0.1){$\vee$}
\textcolor{red}{\put(3,0){\line(1,0){5}}
\put(8,0){\circle*{2}}
\put(3,0){\circle*{2}}}
\end{picture}}

\newcommand{\hquatrehuit}{\begin{picture}(18,8)(-5,-1)
\put(9,7){\circle*{2}}
\put(3,7){\circle*{2}}
\put(2.4,0.1){$\vee$}
\textcolor{red}{\put(0,0){\line(1,0){5}}
\put(0,0){\circle*{2}}
\put(6,0){\circle*{2}}}
\end{picture}}

\newcommand{\hquatreneuf}{\begin{picture}(12,12)(-3,-1)
\put(0,10){\circle*{2}}
\put(0,5){\circle*{2}}
\put(0,0){\line(0,1){5}}
\put(0,5){\line(0,1){5}}
\textcolor{red}{\put(0,0){\line(1,0){5}}
\put(0,0){\circle*{2}}
\put(5,0){\circle*{2}}}
\end{picture}}

\newcommand{\hquatredix}{\begin{picture}(12,12)(-3,-1)
\put(5,5){\circle*{2}}
\put(5,10){\circle*{2}}
\put(5,0){\line(0,1){5}}
\put(5,5){\line(0,1){5}}
\textcolor{red}{\put(0,0){\line(1,0){5}}
\put(0,0){\circle*{2}}
\put(5,0){\circle*{2}}}
\end{picture}}

\newcommand{\hquatreonze}{\begin{picture}(18,8)(-5,-1)
\put(-0.6,0.1){$\vee$}
\textcolor{red}{\put(3,0){\line(1,0){5}}
\put(3,0){\circle*{2}}
\put(8,0){\circle*{2}}}
\textcolor{blue}{\put(-3.5,7){\line(1,0){5}}
\put(2.5,7){\circle*{2}}
\put(-3.5,7){\circle*{2}}}
\end{picture}}

\newcommand{\hquatredouze}{\begin{picture}(12,8)(-3,-1)
\put(2.4,0.1){$\vee$}
\textcolor{red}{\put(0,0){\line(1,0){5}}
\put(6,0){\circle*{2}}
\put(0,0){\circle*{2}}}
\textcolor{blue}{\put(-.5,7){\line(1,0){5}}
\put(5.5,7){\circle*{2}}
\put(-.5,7){\circle*{2}}}
\end{picture}}

\newcommand{\hquatretreize}{\begin{picture}(12,8)(-3,-1)
\put(0,5){\circle*{2}}
\put(5,5){\circle*{2}}
\put(5,0){\line(0,1){5}}
\put(0,0){\line(0,1){5}}
\textcolor{red}{\put(0,0){\line(1,0){5}}
\put(0,0){\circle*{2}}
\put(5,0){\circle*{2}}}
\end{picture}}

\newcommand{\hquatrequatorze}{\begin{picture}(17,8)(-3,-1)
\put(0,5){\circle*{2}}
\put(0,0){\line(0,1){5}}
\textcolor{red}{\put(0,0){\line(1,0){10}}
\put(0,0){\circle*{2}}
\put(5,0){\circle*{2}}
\put(10,0){\circle*{2}}}
\end{picture}}

\newcommand{\hquatrequinze}{\begin{picture}(17,8)(-2,-1)
\put(5,5){\circle*{2}}
\put(5,0){\line(0,1){5}}
\textcolor{red}{\put(0,0){\line(1,0){10}}
\put(0,0){\circle*{2}}
\put(5,0){\circle*{2}}
\put(10,0){\circle*{2}}}
\end{picture}}

\newcommand{\hquatreseize}{\begin{picture}(17,8)(-3,-1)
\put(10,5){\circle*{2}}
\put(10,0){\line(0,1){5}}
\textcolor{red}{\put(0,0){\line(1,0){10}}
\put(0,0){\circle*{2}}
\put(5,0){\circle*{2}}
\put(10,0){\circle*{2}}}
\end{picture}}

\newcommand{\hquatredixsept}{\begin{picture}(21,8)(-3,-1)
\textcolor{red}{\put(0,0){\circle*{2}}
\put(5,0){\circle*{2}}
\put(10,0){\circle*{2}}
\put(15,0){\circle*{2}}
\put(0,0){\line(1,0){15}}}
\end{picture}}

\newcommand{\hddeux}[2]{\begin{picture}(16,8)(-5,-1)
\textcolor{red}{\put(0,0){\circle*{2}}
\put(5,0){\circle*{2}}
\put(0,0){\line(1,0){5}}}
\put(-5,-2){\tiny #1}
\put(7,-2){\tiny #2}
\end{picture}}

\newcommand{\hdtroisun}[3]{\begin{picture}(20,12)(-5,-1)
\put(3,0){\circle*{2}}
\put(-0.6,0.1){$\vee$}
\put(5,-2){\tiny #1}
\put(9,5){\tiny #2}
\put(-5,5){\tiny #3}
\textcolor{red}{\put(6,7){\circle*{2}}
\put(0,7){\circle*{2}}
\put(0,7){\line(1,0){5}}}
\end{picture}}

\newcommand{\hdtroisdeux}[3]{\begin{picture}(16,8)(-5,-1)
\put(0,5){\circle*{2}}
\put(0,0){\line(0,1){5}}
\put(7,-2){\tiny #1}
\put(-5,-2){\tiny #2}
\put(-5,5){\tiny #3}
\textcolor{red}{\put(0,0){\circle*{2}}
\put(5,0){\circle*{2}}
\put(0,0){\line(1,0){5}}}
\end{picture}}

\newcommand{\hdtroisquatre}[3]{\begin{picture}(22,0)(-5,-1)
\put(12,-2){\tiny #1}
\put(3,2){\tiny #2}
\put(-5,-2){\tiny #3}
\textcolor{red}{\put(0,0){\circle*{2}}
\put(5,0){\circle*{2}}
\put(10,0){\circle*{2}}
\put(0,0){\line(1,0){10}}}
\end{picture}}

\newcommand{\hdquatretrois}[4]{\begin{picture}(20,12)(-5,-1)
\put(3,0){\circle*{2}}
\put(-0.5,0.1){$\vee$}
\put(3,0){\line(0,1){7}}
\put(5,-2){\tiny #1}
\put(8.5,5){\tiny #2}
\put(1,10){\tiny #3}
\put(-5,5){\tiny #4}
\textcolor{red}{\put(0,7){\line(1,0){5}}
\put(6,7){\circle*{2}}
\put(0,7){\circle*{2}}
\put(3,7){\circle*{2}}}
\end{picture}}

\newcommand{\hdquatrecinq}[4]{\begin{picture}(20,20)(-5,-1)
\put(3,0){\circle*{2}}
\put(6,14){\circle*{2}}
\put(-0.5,0.1){$\vee$}
\put(6,7){\line(0,1){7}}
\put(5,-2){\tiny #1}
\put(9,5){\tiny #2}
\put(-5,5){\tiny #4}
\put(9,12){\tiny #3}
\textcolor{red}{\put(0,7){\line(1,0){5}}
\put(6,7){\circle*{2}}
\put(0,7){\circle*{2}}}
\end{picture}}

\newcommand{\hdquatresix}[4]{\begin{picture}(20,14)(-5,-1)
\put(3,5){\circle*{2}}
\put(3,0){\circle*{2}}
\put(-0.5,5.1){$\vee$}
\put(3,0){\line(0,1){5}}
\put(6,-3){\tiny #1}
\put(6,4){\tiny #2}
\put(9,12){\tiny #3}
\put(-5,12){\tiny #4}
\textcolor{red}{\put(0,12){\line(1,0){5}}
\put(6,12){\circle*{2}}
\put(0,12){\circle*{2}}}
\end{picture}}

\newcommand{\hdquatresept}[4]{\begin{picture}(20,8)(-5,-1)
\put(3,0){\circle*{2}}
\put(6,7){\circle*{2}}
\put(0,7){\circle*{2}}
\put(8,0){\circle*{2}}
\put(-0.6,0.1){$\vee$}
\put(3,0){\line(1,0){5}}
\put(-3,-2){\tiny #1}
\put(-5,5){\tiny #2}
\put(8,5){\tiny #3}
\put(10,-2){\tiny #4}
\end{picture}}

\newcommand{\hdquatredix}[4]{\begin{picture}(17,12)(-5,-1)
\put(5,5){\circle*{2}}
\put(5,10){\circle*{2}}
\put(5,0){\line(0,1){5}}
\put(5,5){\line(0,1){5}}
\put(-5,-2){\tiny #1}
\put(7,-2){\tiny #2}
\put(7,3){\tiny #3}
\put(7,8){\tiny #4}
\textcolor{red}{\put(0,0){\line(1,0){5}}
\put(0,0){\circle*{2}}
\put(5,0){\circle*{2}}}
\end{picture}}

\newcommand{\hdquatredouze}[4]{\begin{picture}(19,8)(-5,-1)
\put(2.4,0.1){$\vee$}
\put(-5,-2){\tiny #1}
\put(7,-2){\tiny #2}
\put(-2,5){\tiny #3}
\put(10,5){\tiny #4}
\textcolor{red}{\put(0,0){\line(1,0){5}}
\put(6,0){\circle*{2}}
\put(0,0){\circle*{2}}}
\textcolor{blue}{\put(-.5,7){\line(1,0){5}}
\put(5.5,7){\circle*{2}}
\put(-.5,7){\circle*{2}}}
\end{picture}}

\newcommand{\hdquatretreize}[4]{\begin{picture}(18,8)(-5,-1)
\put(0,0){\circle*{2}}
\put(0,5){\circle*{2}}
\put(5,0){\circle*{2}}
\put(5,5){\circle*{2}}
\put(0,0){\line(1,0){5}}
\put(0,0){\line(0,1){5}}
\put(5,0){\line(0,1){5}}
\put(-5,-2){\tiny #1}
\put(-5,3){\tiny #2}
\put(7,-2){\tiny #3}
\put(7,3){\tiny #4}
\end{picture}}

\newcommand{\hdquatreseize}[4]{\begin{picture}(21,8)(-5,-1)
\put(10,5){\circle*{2}}
\put(10,0){\line(0,1){5}}
\put(-5,-2){\tiny #1}
\put(2,2){\tiny #2}
\put(12,-2){\tiny #3}
\put(12,4){\tiny #4}
\textcolor{red}{\put(0,0){\line(1,0){10}}
\put(0,0){\circle*{2}}
\put(5,0){\circle*{2}}
\put(10,0){\circle*{2}}}
\end{picture}}

\newcommand{\hdquatredixsept}[4]{\begin{picture}(24,8)(-5,-1)
\put(-5,-2){\tiny #1}
\put(2,2){\tiny #2}
\put(9,2){\tiny #3}
\put(17,-2){\tiny #4}
\textcolor{red}{\put(0,0){\circle*{2}}
\put(5,0){\circle*{2}}
\put(10,0){\circle*{2}}
\put(15,0){\circle*{2}}
\put(0,0){\line(1,0){15}}}
\end{picture}}

\input{xy}
\xyoption{all}

\title{The Hopf algebra of Fliess operators and its dual prelie algebra }
\date{}
\author{Lo{\"\i}c Foissy \\
\\
{\small{\it Laboratoire de Mathématiques Pures et Appliquées Joseph Liouville}}\\
\small{\it{Université du Littoral Côte d'Opale, Centre Universitaire de la Mi-Voix}}\\
\small{{\it 50, rue Ferdinand Buisson, CS 80699}}\\
\small{{\it 62228 Calais Cedex - France}}\\
\small{e-mail : foissy@lmpa.univ-littoral.fr}}

\newtheorem{defi}{\indent Definition}
\newtheorem{lemma}[defi]{\indent Lemma}
\newtheorem{cor}[defi]{\indent Corollary}
\newtheorem{theo}[defi]{\indent Theorem}
\newtheorem{prop}[defi]{\indent Proposition}

\newenvironment{proof}{{\bf Proof.}}{\hfill $\Box$}

\def\shuff#1#2{\mathbin{
      \hbox{\vbox{\hbox{\vrule \hskip#2 \vrule height#1 width 0pt}\hrule}\vbox{\hbox{\vrule \hskip#2 \vrule height#1 width 0pt\vrule }\hrule}}}}
\def\shuffl{{\mathchoice{\shuff{7pt}{3.5pt}}{\shuff{6pt}{3pt}}{\shuff{4pt}{2pt}}{\shuff{3pt}{1.5pt}}}}
\def\shuffle{\, \shuffl \,}

\newcommand{\K}{\mathbb{K}}
\newcommand{\A}{\K\langle\langle x_0,x_1\rangle\rangle}
\newcommand{\B}{\K\langle x_0,x_1\rangle}
\newcommand{\tcirc}{\tilde{\circ}}
\newcommand{\tdelta}{\tilde{\Delta}}
\newcommand{\T}{\mathcal{T}}
\newcommand{\D}{\mathcal{D}}
\newcommand{\PT}{\mathcal{PT}}
\newcommand{\g}{\mathfrak{g}}
\renewcommand{\PT}{\mathcal{PT}}
\newcommand{\N}{\mathbb{N}}
\newcommand{\Adm}{\mathcal{A}dm}

\begin{document}

\maketitle

ABSTRACT. We study the Hopf algebra $H$ of Fliess operators coming from Control Theory  in the one-dimensional case. 
We prove that it admits a graded, finite-dimensional, connected gradation. Dually, the vector space $\mathbb{R}\langle x_0,x_1\rangle$ is both 
a prelie algebra for the prelie product dual to the coproduct of $H$, and an associative, commutative algebra for the shuffle product. 
These two structures admit a compatibility which makes $\mathbb{R}\langle x_0,x_1\rangle$ a Com-Prelie algebra. 
We give a presentation of this object as a Com-Prelie algebra and as a prelie algebra. \\


AMS CLASSIFICATION. 16T05, 17B60, 93B25, 05C05.

\tableofcontents

\section*{Introduction}

Right prelie algebras, or shortly prelie algebras \cite{Gerstenhaber,Chapoton}, are vector spaces with a bilinear product $\bullet$ 
satisfying the following axiom:
$$(x\bullet y)\bullet z-x\bullet (y\bullet z)=(x\bullet z)\bullet y-x\bullet (z\bullet y).$$
Consequently, the antisymmetrization of $\bullet$ is a Lie bracket. These objects are also called right-symmetric algebras or Vinberg algebra 
\cite{Matsushima,Vinberg}. If $A$ is a prelie algebra, the symmetric algebra $S(A)$ inherits a product $\star$ making it a Hopf algebra,
isomorphic to the enveloping algebra of the Lie algebra $A$ \cite{Oudom1,Oudom2}. Whenever it is possible, we can consider the dual Hopf algebra
$S(A)^*$ and its group of characters $G$, which is the exponentiation, in some sense, of the Lie algebra $A$.\\

We here consider an inverse construction, departing from a group used in Control Theory, namely the group of Fliess operators 
\cite{Ferfera,Gray,Gray2}; this group is used to study the feedback product. We limit ourselves here to the one-dimensional case. This group is the set
$\mathbb{R}\langle\langle x_0,x_1\rangle\rangle$ of noncommutative formal series in two indeterminates, with a certain product generalizing 
the composition  of formal series (definition \ref{1}).  The Hopf algebra $H$ of coordinates of this group is described in \cite{Gray}, 
where it is also proved that it is graded by the length of words;  note that this gradation is not connected and not finite-dimensional. 
We first give a way to describe the composition in the group $\mathbb{R}\langle\langle x_0,x_1\rangle\rangle$ and the coproduct of $H$ 
by induction on the length of words (lemma \ref{2} and proposition \ref{3}). We prove that $H$ admits a second gradation, which is connected; 
the dimensions of this gradation are given by the Fibonacci sequence (proposition \ref{8}). As the product of 
$\mathbb{R}\langle\langle x_0,x_1\rangle\rangle$ is left-linear, $H$ is a commutative, right-sided combinatorial Hopf algebra \cite{Loday},
so, dually, $\mathbb{R}\langle x_0,x_1\rangle$ inherits a prelie product $\bullet$, which is inductively defined in proposition \ref{11}. We prove that
the words $x_1^n$, $n\geq 0$, form a minimal subset of generators of this prelie algebra (theorem \ref{12}).\\

The prelie algebra $\mathbb{R}\langle x_0,x_1\rangle$ has also an associative, commutative product, namely the shuffle product $\shuffle$ 
\cite{Reutenauer}.  We prove that the following axiom is satisfied (proposition \ref{14}):
$$(x\shuffle y)\bullet z=(x\bullet z)\shuffle y+x\shuffle (y\bullet z).$$
So $\mathbb{R}\langle x_0,x_1\rangle$ is a Com-Prelie algebra  \cite{Mansuy} (definition \ref{15}). We give a presentation of this
Com-Prelie algebra in theorem \ref{27}. We use for this a description of free Com-Prelie algebras in terms of partitioned trees 
(definition \ref{17}), which generalizes the construction of prelie algebras in terms of rooted trees  in \cite{Chapoton}. 
We deduce a presentation of $\mathbb{R}\langle x_0,x_1\rangle$ as a prelie algebra in theorem \ref{30}.
This presentation induces a new basis of $\mathbb{R}\langle x_0,x_1\rangle$ in terms of words with letters in $\mathbb{N}^*$, 
see corollary \ref{31}. The prelie product of two elements of this basis uses a dendriform structure \cite{Eilenberg,Loday2}
on the algebra of words with letters in $\mathbb{N}^*$ (theorem \ref{34}). The study of this dendriform structure is postponed to the appendix, 
as well as the enumeration of partitioned trees; we also prove that free Com-Prelie algebras are free as prelie algebras, using Livernet's rigidity theorem 
\cite{Livernet}.\\

{\bf Aknowledgment.} The research leading these results was partially supported by the French National Research Agency under the reference
ANR-12-BS01-0017.\\

{\bf Notation.} We denote by $\K$ a commutative field of characteristic zero. All the objects (algebra, coalgebras, prelie algebras$\ldots$)
in this text will be taken over $\K$.

\section{Construction of the Hopf algebra}

\subsection{Definition of the composition}

Let us consider an alphabet of two letters $x_0$ and $x_1$. We denote by $\A$ the completion of the free algebra generated by this alphabet,
that is to say the set of noncommutative formal series in $x_0$  and $x_1$.
Note that $\A$ is an algebra for the concatenation product and for the shuffle product, which we denote by $\shuffle$. \\

{\bf Exemples.} If $a,b,c,d \in \{x_0,x_1\}$:
\begin{align*}
abc\shuffle d&=abcd+abdc+adbc+dabc,\\
ab\shuffle cd&=abcd+acbd+cabd+acdb+cadb+cdab,\\
a\shuffle bcd&=abcd+bacd+bcad+bcda.
\end{align*}

The unit for both these products is the empty word, which we denote by $\emptyset$. The algebra $\A$ is given its usual ultrametric topology.

\begin{defi}\label{1}\cite{Gray}. \begin{enumerate}
\item For any $d\in \A$, we define a continuous algebra map $\varphi_d$ from $\A$ to $End(\A)$ in the following way: for all $X \in \A$,
$$\varphi_d(x_0)(X)=x_0X,\: \varphi_d(x_1)(X)=x_1X+x_0(d\shuffle X).$$
\item We define a composition $\circ$ on $\A$ in the following way: for all $c,d \in \A$, $c \circ d=\varphi_d(c)(\emptyset)+d$.
\end{enumerate}\end{defi}

It is proved in \cite{Gray} that this composition is associative. \\

{\bf Notation}. For all $c,d\in \A$, we put $c \tcirc d=c\circ d-d=\varphi_d(c)(\emptyset)$. \\

{\bf Remark.} If $c_1,c_2,d \in \A$, $\lambda \in \K$:
$$(c_1+\lambda c_2)\tcirc d=\varphi_d(c_1+\lambda c_2)(\emptyset)=(\varphi_d(c_1)+\lambda \varphi_d(c_2))(\emptyset)
=\varphi_d(c_1)(\emptyset)+\lambda\varphi_d(c_2)(\emptyset)=c_1\tcirc d+\lambda c_2\tcirc d.$$
So the composition $\tcirc$ is linear on the left. As $\varphi_d$ is continuous, the map $c\longrightarrow c\tcirc d$ is continuous for any $d\in \A$. 
Hence, it is enough to know how to compute $c \tcirc d$ for any word $c$, which is made possible by the next lemma, using an induction on the length:

\begin{lemma}\label{2}
For any word $c$, for any $d \in \A$:
\begin{enumerate}
\item $\emptyset \tcirc d=\emptyset$.
\item $(x_0c)\tcirc d=x_0 (c \tcirc d)$.
\item $(x_1c)\tcirc d=x_1(c \tcirc d)+x_0(d \shuffle(c \tcirc d))$.
\end{enumerate}\end{lemma}

\begin{proof} 1. $\emptyset \tcirc d=\varphi_d(\emptyset)(\emptyset)=Id(\emptyset)=\emptyset$. \\

2. $(x_0c)\tcirc d=\varphi_d(x_0c)(\emptyset)=\varphi_d(x_0)\circ \varphi_d(c)(\emptyset)=
\varphi_d(x_0)(c \tcirc d)=x_0(c \tcirc d)$.\\

3. $(x_1c)\tcirc d=\varphi_d(x_1c)(\emptyset)=\varphi_d(x_1)\circ \varphi_d(c)(\emptyset)=
\varphi_d(x_1)(c \tcirc d)=x_1(c \tcirc d)+x_0 (d \shuffle (c \tcirc d))$.  \end{proof}

\subsection{Dual Hopf algebra}

We here give a recursive description of the Hopf algebra of the coordinates of the group $\A$ of \cite{Gray}. \\

For any word $c$, let us consider the map $X_c \in \A^*$, such that for any $d \in \A$, $X_c(d)$ is the coefficient of $c$ in $d$.
We denote by $V$ the subspace of $A^*$ generated by these maps.
Let $H=S(V)$, or equivalently the free associative, commutative algebra generated by the $X_c$'s. The elements of $H$ are seen 
as polynomial functions on $\A$; the elements of $H\otimes H$ are seen as polynomial functions on $\A \times \A$.  Then $H$ is given a multiplicative
coproduct defined in the following way: for any word $c$, for any $f,g \in \A$, 
$$\Delta(X_c)(f, g)=X_c(f\circ g).$$
As $\circ$ is associative, $\Delta$ is coassociative, so $H$ is a bialgebra. \\

{\bf Notations.}\begin{enumerate}
\item The space of words is a commutative algebra for the shuffle product $\shuffle$. Dually, the space $V$ inherits a coassociative,
cocommutative coproduct, denoted by $\Delta_{\shuffle}$. For example, if $a,b,c \in \{x_0,x_1\}$:
\begin{align*}
\Delta_{\shuffle}(X_\emptyset)&=X_\emptyset\otimes X_\emptyset,\\
\Delta_{\shuffle}(X_a)&=X_a\otimes X_\emptyset+X_\emptyset\otimes X_a,\\
\Delta_{\shuffle}(X_{ab})&=X_{ab}\otimes X_\emptyset+X_a\otimes X_b+X_b \otimes X_a+X_\emptyset\otimes X_{ab},\\
\Delta_{\shuffle}(X_{abc})&=X_{abc}\otimes X_\emptyset+X_a\otimes X_{bc}+X_b\otimes X_{ac}+X_c\otimes X_{ab}\\
&+X_{ab}\otimes X_c+X_{ac}\otimes X_b+X_{bc}\otimes X_a+X_\emptyset\otimes X_{abc}.
\end{align*}
\item We define two linear endomorphisms $\theta_0,\theta_1$ of $V$ by $\theta_i(X_c)=X_{x_ic}$ for any word $c$.
\end{enumerate}

The following proposition allows to compute $\Delta(X_c)$ for any word $c$ by induction on the length.

\begin{prop}\label{3}
For all $x \in V$, we put $\tdelta(x)=\Delta(x)-1\otimes x$.
\begin{enumerate}
\item $\tdelta(X_\emptyset)=X_\emptyset \otimes 1$.
\item $\tdelta \circ \theta_0=(\theta_0\otimes Id) \circ \tdelta+(\theta_1 \otimes m)\circ (\tdelta \otimes Id) \circ \Delta_{\shuffle}$.
\item $\tdelta \circ \theta_1=(\theta_1\otimes Id)\circ \tdelta$.
\end{enumerate}\end{prop}

\begin{proof} For any word $c$, for any $f,g \in \A$:
$$\tdelta(X_c)(f,g)=\Delta(X_c)(f,g)-(1\otimes X_c)(f,g)=X_c(f\circ g)-X_c(g)=X_c(f\otimes g-g)=X_c(f\tcirc g).$$
As $\tcirc$ is linear on the left, $\tdelta(X_c) \in V \otimes H$, so  formulas in points 2 and 3 make sense.  \\

Let $f \in \A$. It can be uniquely written as $f=x_0f_0+x_1f_1+\lambda \emptyset$, with $f_0,f_1 \in \A$, $\lambda \in K$. For all $g\in \A$:
$$f\tcirc g=(x_0f_0)\tcirc g+(x_1f_1)\tcirc g+\lambda \emptyset\tcirc g=x_0 (f_0\tcirc g+g \shuffle (f_1\tcirc g))+x_1(f_1 \tcirc g)+\lambda \emptyset.$$

1. We obtain:
 $$\tdelta(X_\emptyset)(f,g)=X_\emptyset(x_0 (f_0\tcirc g+g \shuffle (f_1\tcirc g))+x_1(f_1 \tcirc g)+\lambda \emptyset)
=0+0+\lambda
=(X_\emptyset \otimes 1)(f,g),$$
so $\Delta(X_\emptyset)=X_\emptyset \otimes 1$.\\

2. Let $c$ be a word.
\begin{align*}
\tdelta \circ \theta_0(X_c)(f,g)&=\tdelta(X_{x_0c})(f,g)\\
&=X_{x_0c}(x_0 (f_0\tcirc g+g \shuffle (f_1\tcirc g))+x_1(f_1 \tcirc g)+\lambda \emptyset)\\
&=X_c(f_0 \tcirc g+g\shuffle (f_1\tcirc g))+0+0\\
&=X_c(f_0 \tcirc g+(f_1\tcirc g)\shuffle g)+0+0\\
&=\tdelta(X_c)(f_0,g)+(\tdelta\otimes Id)\circ \Delta_{\shuffle}(X_c)(f_1,g,g)\\
&=\tdelta(X_c)(f_0,g)+(Id \otimes m)\circ(\tdelta\otimes Id)\circ \Delta_{\shuffle}(X_c)(f_1,g)\\
&=(\theta_0\otimes Id) \circ \tdelta(X_c)(f,g)+(\theta_1\otimes Id)\circ (Id \otimes m)\circ(\tdelta\otimes Id)
\circ \Delta_{\shuffle}(X_c)(f,g),
\end{align*}
so $\tdelta \circ \theta_0(X_c)=(\theta_0\otimes Id) \circ \tdelta(X_c)+(\theta_1\otimes Id)\circ (Id \otimes m)\circ(\tdelta\otimes Id)
\circ \Delta_{\shuffle}(X_c)$.\\

3. Let $c$ be a word.
\begin{align*}
\tdelta \circ \theta_1(X_c)(f,g)&=\tdelta(X_{x_0c})(f,g)\\
&=X_{x_1c}(x_0 (f_0\tcirc g+g \shuffle (f_1\tcirc g))+x_1(f_1 \tcirc g)+\lambda \emptyset)\\
&=0+X_c(f_1 \tcirc g)+0\\
&=\tdelta(X_c)(f_1 ,g)\\
&=(\theta_1\otimes Id) \circ \tdelta(X_c)(f,g),
\end{align*}
so $\tdelta \circ \theta_1(X_c)=(\theta_1\otimes Id) \circ \tdelta(X_c)$. \end{proof}\\

{\bf Examples.}
\begin{align*}
\Delta(X_{x_0})&=X_{x_0}\otimes 1+1\otimes X_{x_0}+X_{x_1}\otimes X_\emptyset,\\
\Delta(X_{x_0^2})&=X_{x_0^2}\otimes 1+1\otimes X_{x_0^2}+
X_{x_0x_1}\otimes X_\emptyset+X_{x_1x_0}\otimes X_\emptyset+X_{x_1x_1}\otimes X_\emptyset^2+X_{x_1}
\otimes X_{x_0},\\
\Delta(X_{x_0x_1})&=X_{x_0x_1}\otimes 1+1\otimes X_{x_0x_1}+X_{x_1x_1}\otimes X_\emptyset
+X_{x_1}\otimes X_{x_1},\\
\Delta(X_{x_1x_0})&=X_{x_1x_0}\otimes 1+1\otimes X_{x_1x_0}+X_{x_1x_1}\otimes X_\emptyset.
\end{align*}

\begin{cor}\label{4}
For all $n \geq 1$, $\tdelta(X_{x_1^n})=X_{x_1^n} \otimes 1$ and  $\Delta(X_{x_1^n})=X_{x_1^n} \otimes 1
+1\otimes X_{x_1^n}$.
\end{cor}

\begin{proof} It comes from an easy induction on $n$. \end{proof}

\subsection{gradation}

It is proved in \cite{Gray} that the Hopf algebra $H$ is graded by the length of words, but this gradation is not connected, that is to say that 
the homogeneous component of degree $0$ is not $(0)$, as it contains $X_{\emptyset}$. 
Moreover, the homogeneous components of $H$ are not finite-dimensional, as for example $X_{\emptyset}^n X_{x_0^k}$ is homogeneous
of degree $k$ for all $n \geq 0$.  We now define another gradation on $H$, which is connected and finite-dimensional.

\begin{defi}\begin{enumerate}
\item Let $c=c_1\ldots c_n$ be a word. We put:
$$deg(c)=n+1+\sharp\left\{i\in \{1,\ldots,n\}\mid c_i=x_0\right\}.$$
\item For all $k\geq 1$, we put:
$$V_k=Vect(X_c\mid deg(x)=k).$$
This define a connected gradation of $V$, that is to say:
$$V=\bigoplus_{k\geq 1} V_k.$$
\item This gradation induces a connected gradation of the algebra $H$:
$$H=\bigoplus_{k\geq 0} H_k,\mbox{ and } H_0=\K.$$
\end{enumerate}\end{defi}

\begin{lemma}
If $c$ is a word of degree $n$, then:
$$\tdelta(X_c) \in \bigoplus_{i+j=n} V_i \otimes H_j.$$
\end{lemma}

\begin{proof} Let us start by the following observations:
\begin{enumerate}
\item Let $c$ be a word of degree $k$. Then $x_0c$ is a word of degree $k+2$. Hence, $\theta_0$ is homogeneous of degree $2$ on $V$.
\item Let $c$ be a word of degree $k$. Then $x_1c$ is a word of degree $k+1$. Hence, $\theta_1$ is homogeneous of degree $1$ on $V$.
\item Let $c$ and $d$ be two words of respective degrees $k$ and $l$. Then any word obtained by shuffling $c$ and $d$ is of degree $k+l-1$:
its length is the sum of the length of $c$ and $d$, and the number of $x_0$ in it is the sum of the numbers of $x_0$ in $c$ and $d$.
Hence, dually, the coproduct $\Delta_{\shuffle}$ is homogeneous of degree $1$ from $V$ to $V \otimes V$.
\end{enumerate}

Let us prove the result by induction on the length $k$ of $c$. If $k=0$, then $c=\emptyset$ so $n=1$, and 
$\tdelta(X_c)=X_c \otimes 1 \in V_1 \otimes H_0$. Let us assume the result for all words of length $<k-1$. Two cases can occur.
\begin{enumerate}
\item If $c=x_0d$, then $deg(d)=n-2$. we put $\Delta_{\shuffle}(X_d)=\sum x'_i\otimes x''_i$. By the preceding third observation, 
we can assume that for all $i$, $x'_i,x''_i$ are homogeneous elements of $V$, with $deg(x'_i)+deg(x'_i)=n-2+1=n-1$. Then:
$$\tdelta(X_c)=(\theta_0\otimes Id) \circ \tdelta(X_d)+\sum_i(\theta_1 \otimes m)\circ(\tdelta(x'_i) \otimes x''_i).$$
By the induction hypothesis, $\tdelta(X_d) \in (V\otimes H)_{n-1}$. By the second observation, 
$(\theta_0\otimes Id) \circ \tdelta(X_d)\in (V\otimes H)_n$. By the induction hypothesis applied to $x'_i$, for all $i$, 
$(\tdelta(x'_i) \otimes x''_i) \in (V \otimes H\otimes V)_{n-1}$, so by the first observation,
$(\theta_1 \otimes m)\circ(\tdelta(x'_i) \otimes x''_i) \in (V \otimes H)_{n-1+1}\subseteq (V \otimes H)_n$. 
So $\Delta(X_c)\in (V \otimes H)_n$.
\item $c=x_1d$, then $deg(d)=n-1$. Moreover, $\tdelta(X_c)=(\theta_1\otimes Id)\circ \tdelta(X_d)$.
By the induction hypothesis, $\tdelta(X_d) \in (V \otimes H)_{n-1}$. By the second observation, $\tdelta(X_c) \in (V \otimes H)_n$.
\end{enumerate}
So the result holds for any word $c$.  \end{proof}

\begin{prop}
With this gradation, $H$ is a graded, connected Hopf algebra.
\end{prop}

\begin{proof}  We have to prove that for all $n \geq 0$:
$$\Delta(H_n)\subseteq \bigoplus_{i+j=n} H_i\otimes H_j.$$
This comes from the multiplicativity of $\Delta$.  \end{proof}\\

Let us now study the formal series of $V$ and $H$.

\begin{prop}\label{8}\begin{enumerate}
\item For all $k$, let us put $p_k=dim(V_k)$ and $F_V=\displaystyle \sum_{k=1}^\infty p_k X^k$. Then:
$$F_V=\frac{X}{1-X-X^2},$$
and for all $k\geq 1$:
$$p_k=\frac{1}{\sqrt{5}}\left(\left(\frac{1+\sqrt{5}}{2}\right)^k-\left(\frac{1-\sqrt{5}}{2}\right)^k\right).$$
This is the Fibonacci sequence (A000045 in \cite{Sloane}).
\item We put $F_H=\displaystyle \sum_{k=0}^\infty dim(H_k)X^k$. Then:
$$F_H=\prod_{k=1}^\infty  \frac{1}{(1-X^k)^{p_k}}.$$
\end{enumerate}\end{prop}

\begin{proof} Let us consider the formal series:
$$F(X_0,X_1)=\sum_{i,j\geq 0} \sharp\{\mbox{words in $x_0, x_1$ with $i$ $x_0$ and $j$ $x_1$}\} X_0^i X_0^j.$$
Then $\displaystyle F(X_0,X_1)=\frac{1}{1-X_0-X_1}$. Moreover, by definition of the degree of a word:
$$F_V=XF(X^2,X)=\frac{X}{1-X-X^2}.$$
As $H$ is the symmetric algebra generated by $V$, its formal series is given by the second point. \end{proof}\\

{\bf Examples.} We obtain:
$$\begin{array}{c|c|c|c|c|c|c|c|c|c|c|c}
k&0&1&2&3&4&5&6&7&8&9&10\\
\hline dim(V_k)&0&1&1&2&3&5&8&13&21&34&55\\
\hline dim(H_k)&1&1&2&4&8&15&30&56&108&203&384
\end{array}$$
The third row is sequence A166861 of \cite{Sloane}. \\

{\bf Remark.} Consequently, the space $V$ inherits a bigradation:
$$V_{k,n}=Vect(X_c\mid deg(c)=k \mbox{ and }lg(c)=n).$$
If $c$ is a word of length $n$ and of degree $k$, denoting by $a$ the number of its letters equal to $x_0$ and by $b$ the number of its letters equal 
to $x_1$, then:
$$\begin{cases}
a+b&=n,\\
2a+b+1&=k,
\end{cases}$$
so $a=k-n-1$. Hence:
$$dim(V_{k,n})=\binom{n}{k-n-1},$$
and the formal series of this bigradation is:
$$\sum_{k,n\geq 0} dim(V_{k,n})X^kY^n=\frac{X}{1-XY-X^2Y}.$$
\section{Prelie structure on $\B$}

\subsection{Prelie coproduct on $V$}

As the composition $\circ$ is linear on the left, the dual coproduct satisfies $\tdelta(V)\subseteq V\otimes H$, so $H$ is a commutative right-sided 
Hopf algebra in the sense of \cite{Loday}, and $V$ inherits a right prelie coproduct: if $\pi$ is the canonical projection from $H=S(V)$ onto $V$, 
$$\delta=(\pi \otimes \pi) \circ \Delta=(Id \otimes \pi) \circ \tdelta.$$
It satisfies the right prelie coalgebra axiom:
$$(23).((\delta \otimes Id)\circ \delta-(Id \otimes \delta)\circ\delta)=0.$$
The following proposition allows to compute $\delta(X_c)$ by induction on the length of $c$.

\begin{prop}\label{9}\begin{enumerate}
\item $\delta(X_\emptyset)=0$.
\item $\delta \circ \theta_0=(\theta_0 \otimes Id)\circ \delta+(\theta_1\otimes Id)\circ \Delta_{\shuffle}$.
\item $\delta \circ \theta_1=(\theta_1 \otimes Id)\circ \delta$.
\end{enumerate}\end{prop}

\begin{proof} The first point comes from $\Delta(X_\emptyset)=X_\emptyset\otimes 1+1\otimes X_\emptyset$. Let $x \in V$. We put
$\Delta_{\shuffle}(x)=x'\otimes x'' \in V \otimes V$.  For any $y\in V$, we put $\tdelta(y)-y\otimes 1
=y^{(1)}\otimes y^{(2)} \in V\otimes H_+$. Then:
\begin{align*}
(\theta_1 \otimes m)\circ (\tdelta \otimes Id) \circ \Delta_{\shuffle}(x)
&=(\theta_1\otimes m)(x'\otimes 1\otimes x''+x'^{(1)}\otimes x'^{(2)}\otimes x'')\\
&=\theta_1(x') \otimes \underbrace{x''}_{\in V}+x'^{(1)}\otimes \underbrace{x'^{(2)} x''}_{\in Ker(\pi)}.
\end{align*}
Applying $Id \otimes \pi$, it remains:
$$(Id \otimes \pi)\circ(\theta_1 \otimes m)\circ (\tdelta \otimes Id) \circ \Delta_{\shuffle}(x)
=(\theta_1 \otimes Id)\circ \Delta_{\shuffle}(x).$$
Let $i=0$ or $1$. Then:
$$(Id \otimes \pi) \circ (\theta_i\otimes Id)\circ \tdelta=(\theta_i\otimes Id)\circ (Id \otimes \pi)\circ \tdelta
=(\theta_i\otimes Id)\circ \delta.$$
The result is induced by these remarks, combined with proposition \ref{3}. \end{proof}\\

{\bf Examples.}
\begin{align*}
\delta(X_{x_0})&=X_{x_1}\otimes X_\emptyset,\\
\delta(X_{x_0^2})&=X_{x_0x_1}\otimes X_\emptyset+X_{x_1x_0}\otimes X_\emptyset+X_{x_1}\otimes X_{x_0},\\
\delta(X_{x_0x_1})&=X_{x_1x_1}\otimes X_\emptyset+X_{x_1}\otimes X_{x_1},\\
\delta(X_{x_1x_0})&=X_{x_1x_1}\otimes X_\emptyset.
\end{align*}

\begin{prop}\label{10}
$Ker(\delta)=Vect(X_{x_1^n},n\geq 0)$.
\end{prop}

\begin{proof} The inclusion $\supseteq$ is trivial by corollary \ref{4}. Let us prove the other inclusion. \\

{\it First step.} Let us prove the following property: if $x \in V_k$ is such that
$$\delta(x)=\lambda \sum_{i+j=k-2} \frac{(k-2)!}{i!j!}X_{x_1^i}\otimes X_{x_1^j},$$
then there exists $\mu \in \K$ such that $x=\mu x_1^{k-1}$. It is obvious if $k=1$, as then $x=\mu \emptyset$. Let us assume 
the result at all ranks $<k$. We put $x=x_1^\alpha(x_0f_0+x_1f_1)$, where $\alpha\geq 0$, $f_0$ is homogeneous of degree $k-2-\alpha$ 
and $f_1$ is homogeneous of degree $k-1-\alpha$.
$$\delta(x)=(\theta_1^\alpha\otimes Id)\left((\theta_0\otimes Id)\circ \delta(f_0)+(\theta_1\otimes Id)\circ \delta(f_1)
+(\theta_1\otimes Id)\circ \Delta_{\shuffle}(f_0)\right).$$
Let us consider the terms in this expression of the form $X_\emptyset \otimes X_c$, with $c$ a word. This gives:
$$\lambda X_\emptyset\otimes X_{x_1^{k-2}}=0,$$
so $\lambda=0$ and $\delta(x)=0$.
Let us now consider the terms of the form $X_{x_1^\alpha x_0c}\otimes X_d$, with $c,d$ words. We obtain:
$$(\theta_1^\alpha \circ \theta_0 \otimes Id)\circ \delta(f_0)=0.$$
As both $\theta_0$ and $\theta_1$ are injective, we obtain $\delta(f_0)=0$. By the induction hypothesis, $f_0=\nu X_1{x_1^l}$, 
with $l=k-2-\alpha<k$. Hence:
$$\Delta_{\shuffle}(f_0)=\nu \sum_{i+j=l}\frac{l!}{i!j!}X_{x_1^i} \otimes X_{x_1^j},$$
and:
$$(\theta_1^{\alpha+1}\otimes Id)\left(\delta(f_1)+\nu \sum_{i+j=l} \frac{l!}{i!j!}X_{x_1^i} \otimes X_{x_1^j}\right)=0.$$
As $\theta_1$ is injective, we obtain with the induction hypothesis that $f_1=\mu X_{x_1^{k-2-\alpha}}$, so:
$$x=\mu X_{x_1^{k-1}}+\nu X_{x_1^\alpha x_0 x_1^{k-\alpha-2}}.$$
This gives:
\begin{align*}
\delta(x)&=\nu (\theta_1^{\alpha+1}\otimes Id)\left(\sum_{i+j=k-\alpha-2} \frac{(k-\alpha-2)!}{i!j!}
X_{x_1^i}\otimes X_{x_1^j}\right)\\
&=\nu \sum_{i+j=k-\alpha-2} \frac{(k-\alpha-2)!}{i!j!}
X_{x_1^{i+\alpha}}\otimes X_{x_1^j}
\\
&=0,
\end{align*}
so necessarily $\nu=0$ and $x=\mu X_{x_1^{k-1}}$.\\

{\it Second step.} Let $x\in Ker(\delta)$. As $\delta$ is homogeneous of degree $0$, the homogeneous components of $x$ are in $Ker(\delta)$.
By the first step, with $\lambda=0$, these homogeneous components, hence $x$, belong to $Vect(X_{x_1^k},k\geq 0)$. \end{proof}

\subsection{Dual prelie algebra}

As $V$ is a graded prelie coalgebra, its graded dual is a prelie algebra. We identify this graded dual with $\B\subseteq \A$; for any words $c,d$,
$X_c(d)=\delta_{c,d}$. The prelie product of $\B$ is denoted by $\bullet$.  Dualizing proposition \ref{9}, we obtain:

\begin{prop}\label{11}\begin{enumerate}
\item For all word $c$, $\emptyset \bullet c=0$.
\item For all words $c,d$, $(x_0c)\bullet d=x_0(c\bullet d)$.
\item For all words $c,d$, $(x_1c)\bullet d=x_1(c \bullet d)+x_0(c \shuffle d)$.
\end{enumerate}\end{prop}

\begin{proof} Let $u,v,w$ be words. Then $X_w(u\bullet v)=\delta(X_w)(u\otimes v)$. Hence, if $d$ is a word:
\begin{align*}
X_\emptyset(u\bullet v)&=0,\\
X_{x_0d}(u\bullet v)&=(\theta_0\otimes Id)\circ \delta(X_d)(u\otimes v)+(\theta_1\otimes Id)\circ \Delta_{\shuffle}(X_d)
(u\otimes v)\\
&=X_d(\theta_0^*(u)\bullet v+\theta_1^*(u)\shuffle v),\\
X_{x_1d}(u\bullet v)&=(\theta_1\otimes Id)\otimes \delta(X_d)(u\otimes v)\\
&=X_d(\theta_1^*(u)\bullet v).
\end{align*}
Moreover, for all word $c$:
\begin{align*}
\theta_0^*(\emptyset)&=0,&\theta_0^*(x_0c)&=c,&\theta_0^*(x_1c)&=0,\\
\theta_1^*(\emptyset)&=0,&\theta_1^*(x_0c)&=0,&\theta_1^*(x_1c)&=c.
\end{align*}
Hence, for any words $c,d$:
\begin{align*}
X_{x_0d}(x_0c \bullet v)&=X_d(c \bullet v)&X_{x_0d}(x_1c \bullet v)&=X_d(c \shuffle v)\\
&=X_{x_0d}(x_0(x \bullet v)),&&=X_{x_0d}(x_1(c \bullet v)+x_0(c \shuffle v)),\\ \\
X_{x_1d}(x_0c \bullet v)&=0&X_{x_1d}(x_1c \bullet v)&=X_d(c \bullet v)\\
&=X_{x_1d}(x_0(x \bullet v)),&&=X_{x_1d}(x_1(c \bullet v)+x_0(c \shuffle v)).
\end{align*}
Hence, for any $w$, $X_w(x_0c\bullet v)=X_w(x_0(x \bullet v))$ and 
$X_w(x_1c\bullet v)=X_w((x_1(c \bullet v)+x_0(c \shuffle v))$. \end{proof}\\

{\bf Examples.}
\begin{align*}
x_0\bullet x_0&=0&x_0\bullet x_0x_0&=0&x_1\bullet x_0x_0&=x_0x_0x_0\\
x_0\bullet x_1&=0&x_0\bullet x_0x_1&=0&x_1\bullet x_0x_1&=x_0x_0x_1\\
x_1\bullet x_0&=x_0x_0&x_0\bullet x_1x_0&=0&x_1\bullet x_1x_0&=x_0x_1x_0\\
x_1\bullet x_1&=x_0x_1&x_0\bullet x_1x_1&=0&x_1\bullet x_1x_1&=x_0x_1x_1
\end{align*}
\begin{align*}
x_0x_0\bullet x_0&=0&x_0x_0\bullet x_1&=0\\
x_0x_1\bullet x_0&=x_0x_0x_0&x_0x_1\bullet x_1&=x_0x_0x_1\\
x_1x_0\bullet x_0&=2x_0x_0x_0&x_1x_0\bullet x_1&=x_0x_0x_1+x_0x_1x_0\\
x_1x_1\bullet x_0&=x_1x_0x_0+x_0x_1x_0+x_0x_0x_1&x_1x_1\bullet x_1&=x_1x_0x_1+2x_0x_1x_1
\end{align*}

Dualizing proposition \ref{10}:

\begin{theo}\label{12}
$\B=Vect(x_1^n,n\geq 0)\oplus (\B\bullet \B)$. Hence, $(x_1^n)_{n\geq 0}$ is a minimal system of generators of the prelie algebra $\B$.
\end{theo}

\begin{proof} As $\bullet=\delta^*$, $Im(\bullet)=Ker(\delta)^\perp=Vect(X_{x_1^n},n\geq 0)^\perp$. The first assertion is then immediate.
As $\A$ is a graded, connected prelie coalgebra, $\B$ is a graded, connected prelie algebra. The result then comes from the next lemma. \end{proof}

\begin{lemma}
Let $A$ be a graded, connected prelie algebra, and $V$ be a graded subspace of $A$.
\begin{enumerate}
\item $V$ generates $A$ if, and only if, $A=V+A\bullet A$.
\item $V$ is a minimal subspace of generators of $A$ if, and only if, $A=V\oplus A\bullet A$.
\end{enumerate}\end{lemma}

\begin{proof} 1. $\Longrightarrow$. Let $x \in A$. Then it can be written as an element of the prelie subalgebra generated by $v$, 
so as the sum of an element of $V$ and of iterated prelie products of elements of $V$. Hence, $x\in V+A\bullet A$.  Note that we did not use 
the gradation of $A$ to prove this point.\\

1. $\Longleftarrow$. Let $B$ be the prelie subalgebra generated by $V$. Let $x\in A_n$, let us prove that $x\in B$ by induction on $n$.
As $A_0=(0)$, it is obvious if $n=0$. Let us assume the result at all ranks $<n$. We obtain, by the gradation:
$$A_n=V_n\oplus \sum_{i=1}^{n-1} A_i \bullet A_{n-i}.$$
So we can write $x=\lambda x_1^{n-1}+\sum x_i \bullet y_i$, 
where $x_i$, $y_i$ are homogeneous of degree $<n$. By the induction hypothesis, these elements belong to $B$, so $x \in B$. \\

2. $\Longrightarrow$. By 1. $\Longrightarrow$, $A=V+A\bullet A$. If $V \cap A\bullet A\neq (0)$, we can choose a graded subspace $W\subsetneq V$,
such that $A=W\oplus A\bullet A$. By 1. $\Longleftarrow$, $W$ generates $A$, so $V$ is not a minimal system of generators of $A$: contradiction.
So $A=V\oplus A\bullet A$. \\

2. $\Longleftarrow$. By 1. $\Longleftarrow$, $V$ is a space of generators of $A$. If $W \subsetneq V$, then $W \oplus A\bullet A \subsetneq A$. 
By 1. $\Longrightarrow$, $W$ does not generate $V$. So $V$ is a minimal subspace of generators.
\end{proof}

\begin{prop}\label{14}
For all $x,y,z\in \B$, $(x\shuffle y)\bullet z=(x\bullet z)\shuffle y+x\shuffle (y\bullet z)$.
\end{prop}

\begin{proof} We prove it if $x,y,z$ are words. If $x=\emptyset$, then:
$$(\emptyset \shuffle y)\bullet z=y\bullet z=(\emptyset \bullet z)\shuffle y+\emptyset\shuffle(u\bullet z).$$
If $y=\emptyset$, the result is also true, using the commutativity of $\shuffle$. We can now consider that $x,y$ are nonempty words.

Let us proceed by induction on $k=lg(x)+lg(y)$.
If $k=0$ or $1$, there is nothing to prove. Let us assume the result at all rank $<k$. Four cases can occur.\\

{\it First case.} $x=x_0a$ and $y=x_0b$. Then:
\begin{align*}
(x\shuffle y)\bullet z&=(x_0(a\shuffle x_0b) \bullet z+(x_0(x_0a\shuffle b))\bullet z\\
&=x_0((a\shuffle x_0b)\bullet z)+x_0((x_0a\shuffle b)\bullet z)\\
&=x_0((a\bullet z)\shuffle x_0b)+x_0(a\shuffle ((x_0b)\bullet z))
+x_0(((x_0a)\bullet z)\shuffle b)+x_0(x_0a\shuffle (b\bullet z))\\
&=x_0((a\bullet z)\shuffle x_0b)+x_0(a\shuffle (x_0(b\bullet z))
+x_0((x_0(a\bullet z))\shuffle b)+x_0(x_0a\shuffle (b\bullet z))\\
&=x_0(a\bullet z)\shuffle x_0b+x_0a\shuffle x_0(b\bullet z)\\
&=(x\bullet z)\shuffle y+x\shuffle (y\bullet z).
\end{align*}

{\it Second case.} $x=x_1a$ and $y=x_0b$. This gives:
\begin{align*}
(x\shuffle y)\bullet z&=(x_1(a\shuffle x_0b))\bullet z+(x_0(x_1a\shuffle b))\bullet z\\
&=x_1((a\bullet z)\shuffle x_0b)+x_1(a\shuffle x_0(b\bullet z))\\
&+x_0(a\shuffle x_0b\shuffle z)+x_0(((x_1a)\bullet z)\shuffle b)+x_0(x_1a\shuffle(b\bullet z))\\
&=x_1((a\bullet z)\shuffle x_0b)+x_1(a\shuffle x_0(b\bullet z))\\
&+x_0(a\shuffle x_0b\shuffle z)+x_0((x_1(a\bullet z))\shuffle b)
+x_0((x_0(a\shuffle z))\shuffle b)+x_0(x_1a\shuffle(b\bullet z)),\\ \\
(x\bullet z)\shuffle y&=(x_1(a\bullet z))\shuffle x_0b+(x_0(a\shuffle z))\shuffle(x_0b)\\
&=x_1((a\bullet z)\shuffle (x_0b))+x_0(x_1(a\bullet z)\shuffle b)\\
&+x_0(a\shuffle z\shuffle x_0b)+x_0((x_0(a\shuffle z))\shuffle b), \\ \\
x\shuffle(y\bullet z)&=x_1a \shuffle x_0(b\bullet z)\\
&=x_1(a\shuffle x_0(b\bullet z))+x_0(x_1a\shuffle (b\bullet z)).
\end{align*}
These computations imply the required equality.\\

{\it Third case.} $x=x_0a$ and $y=x_1b$. This is a consequence of the second case, using the commutativity of $\shuffle$.\\

{\it Last case.} $x=x_1a$ and $y=x_1b$. Similar computations give:
\begin{align*}
(x\shuffle y)\bullet z&=x_1((a\bullet z)\shuffle x_1b)+x_1(a\shuffle x_1(b\bullet w))+x_1(a\shuffle x_0(b\shuffle z))+
x_0(a\shuffle x_1b\shuffle z)\\
&+x_1(x_1a\shuffle(b\bullet z))+x_1((x_1(a\bullet z))\shuffle b)+x_1((x_0(a\shuffle z))\shuffle b)
+x_0(a\shuffle x_1b\shuffle z),\\ \\
(x\bullet z)\shuffle y&=x_1((a\bullet z)\shuffle x_1b)+x_1((x_1(a\bullet z))\shuffle b)+x_0(a\shuffle x_1b\shuffle z)
+x_1((x_0(a\shuffle z))\shuffle b),\\ \\
x\shuffle (y\bullet z)&=x_1(a\shuffle x_1(b\bullet w))+x_1(a\shuffle x_0(b\shuffle z))+x_1(x_1a\shuffle(b\bullet z))
+x_0(a\shuffle x_1b\shuffle z).
\end{align*}
So the result holds in all cases. \end{proof}

\section{Presentation of $\B$ as a Com-Prelie algebra}

Proposition \ref{14} motivates the following definition:

\begin{defi}\label{15} \cite{Mansuy}
A \emph{Com-Prelie algebra} is a triple $(V,\bullet,\shuffle)$, such that:
\begin{enumerate}
\item $(V,\bullet)$ is a prelie algebra.
\item $(V,\shuffle)$ is a commutative, associative algebra (non necessarily unitary).
\item For all $a,b,c \in V$, $(a\shuffle b)\bullet c=(a\bullet c) \shuffle b+a\shuffle (b\bullet c)$.
\end{enumerate}\end{defi}

For example, $\B$ is a Com-Prelie algebra. See \cite{Mansuy} for an example of Com-Prelie algebra based on rooted trees.

\subsection{Free Com-Prelie algebras}

\begin{defi}\begin{enumerate}
\item A \emph{partitioned forest} is a pair $(F,I)$ such that:
\begin{enumerate}
\item $F$ is a rooted forest (the edges of $F$ being oriented from the leaves to the roots).
\item $I$ is a partition of the vertices of $F$ with the following condition:
if $x,y$ are two vertices of $F$ which are in the same part of $I$, then either they are both roots, or they have the same direct descendant.
\end{enumerate}
\item We shall say that a partitioned forest is a \emph{partitioned tree} if all the roots are in the same part of the partition.
\item Let $\D$ be a set. A \emph{partitioned tree decorated by $\D$} is a pair $(t,d)$, where $t$ is a partitioned tree and $d$ 
is a map from the set of vertices of $t$ into $\D$. For any vertex $x$ of $t$, $d(x)$ is called the \emph{decoration} of $x$.
\item The set of isoclasses of partitioned trees will be denoted by $\PT$. 
For any set $\D$, the set of isoclasses of partitioned trees decorated by $\D$ will be denoted by $\PT(\D)$.
\end{enumerate}\end{defi}

{\bf Examples.} We represent partitioned trees by the Hasse graph of the underlying rooted forest, the partition being represented by horizontal edges,
of different colors. Here are all the partitioned trees with $\leq 4$ vertices:
$$\tun;\tdeux,
\hdeux;\ttroisun,
\htroisun,\ttroisdeux,
\htroisdeux=\htroistrois,
\htroisquatre;
\tquatreun, \hquatreun=\hquatredeux,\hquatretrois,\tquatredeux=\tquatretrois,\hquatrequatre=\hquatrecinq,\tquatrequatre,\hquatresix,\tquatrecinq,$$
$$\hquatresept=\hquatrehuit,\hquatreneuf=\hquatredix,\hquatreonze=\hquatredouze,\hquatretreize,\hquatrequatorze=\hquatrequinze=\hquatreseize,
\hquatredixsept.$$

\begin{defi}\label{17} Let $t=(t,I)$ and $t'=(t',J) \in \PT$. 
\begin{enumerate}
\item Let $s$ be a vertex of $t'$. The partitioned tree $t\bullet_s t'$ is defined as follows:
\begin{enumerate}
\item As a rooted forest, $t \bullet_s t'$ is obtained by grafting all the roots of $t'$ on the vertex $s$ of $t$.
\item We put $I=\{I_1,\ldots,I_k\}$ and $J=\{J_1,\ldots,J_l\}$. The partition of the vertices of this rooted forest is 
$I\sqcup J=\{I_1,\ldots,I_k,J_1,\ldots,J_l\}$.
\end{enumerate}
\item The partitioned tree $t\shuffle t'$ is defined as follows:
\begin{enumerate}
\item As a rooted forest, $t \shuffle t'$ is $tt'$.
\item We put $I=\{I_1,\ldots,I_k\}$ and $J=\{J_1,\ldots,J_l\}$
and we assume that the set of roots of $t$ is $I_1$ and the set of roots of $t'$ is $J_1$. The partition of the vertices of $t \shuffle t'$
$\{I_1\sqcup J_1,I_2,\ldots,I_k,J_2,\ldots,J_l\}$.
\end{enumerate}\end{enumerate}\end{defi}

{\bf Examples.} \begin{enumerate}
\item Here are the three possible graftings $\htroisun \bullet_s \tun$: $\hquatreun$, $\hquatrequatre$ and $\hquatrecinq$.
\item Here are the two possible graftings $\tdeux \bullet_s \hdeux$: $\hquatreun$ and $\hquatresix$.
\end{enumerate}

These operations can also be defined for decorated partitioned trees.

\begin{prop}
Let $\D$ be a set. We denote by $\g_{\PT(\D)}$ the vector space generated by $\PT(\D)$. 
We extend $\shuffle$ by bilinearity on $\g_{\PT(\D)}$ and we define a second product $\bullet$ on $\g_{\PT(\D)}$ in the following way:
if $t,t'\in \PT(\D)$, 
$$t\bullet t'=\sum_{s\in V(t)} t\bullet_s t'.$$
Then $(\g_{\PT(\D)}, \bullet,\shuffle)$ is a Com-Prelie algebra. 
\end{prop}

\begin{proof} Let $t,t',t''$ be three partitioned trees.

If $s',s''$ are two vertices of $t$, we define
by $t\bullet_{s,s'}(t',t'')$ the partitioned trees obtained by grafting the roots of $t'$ on $s'$, the roots of $t''$ on $s''$, the partition of the
vertices of the obtained rootes forest being the union of the partitions of $t$, $t'$ and $t''$. Then:
\begin{align*}
(t\bullet t')\bullet t''&=\sum_{s'\in V(t)} (t\bullet_{s'} t') \bullet t''\\
&=\sum_{s',s''\in V(t)}(t\bullet_{s'} t')\bullet_{s''} t''+\sum_{s'\in V(t), s''\in V(t')} (t\bullet_{s'} t')\bullet_{s''} t''\\
&=\sum_{s',s''\in V(t)} t\bullet_{s's''}(t',t'')+\sum_{s'\in V(t), s''\in V(t')} t\bullet_{s'} (t'\bullet_{s''} t'')\\
&=\sum_{s',s''\in V(t)} t\bullet_{s's''}(t',t'')+t\bullet (t'\bullet t'').
\end{align*}
So $(t\bullet t')\bullet t''-t\bullet (t'\bullet t'')$ is clearly symmetric in $t$ and $t'$, and $\bullet$ is prelie. \\

Moreover, $(t\shuffle t') \shuffle t''=t\shuffle (t'\shuffle t'')$ is the rooted forest $tt't''$, the partition being $\{I_1\sqcup J_1 \sqcup K_1,I_2,\ldots,I_k,
J_2,\ldots,J_l,K_2,\ldots,K_m\}$, with immediate notations; $t \shuffle t'=t'\shuffle t$ is the rooted forest $tt'$, the partition being
$\{I_1\sqcup J_1,I_2,\ldots,I_k,J_2,\ldots,J_l\}$. So $\shuffle$ is an associative, commutative product.

Finally:
\begin{align*}
(t\shuffle t')\bullet t''&=\sum_{s\in V(t)} (t\shuffle t') \bullet_s t''+\sum_{s'\in V(t')} (t\shuffle t')\bullet_{s'} t''\\
&=\sum_{s\in V(t)} (t\bullet_s t'')\shuffle t' +\sum_{s'\in V(t')} t\shuffle (t'\bullet_{s'} t'')\\
&=(t\bullet t')\shuffle t''+t\shuffle (t'\bullet t'').
\end{align*}
So $\g_{\PT(\D)}$ is Com-Prelie. \end{proof}\\

In particular, $\g_{\PT(\D)}$ is prelie. Let us use the extension of the prelie product $\bullet$ to $S(\g_{\PT(\D)})$ 
defined by Oudom and Guin \cite{Oudom1,Oudom2}: 
\begin{enumerate}
\item If $t_1,\ldots, t_k \in \g_{\PT(\D)}$, $t_1\ldots t_k \bullet 1=t_1\ldots t_k$.
\item If $t,t_1,\ldots,t_k \in \g_{\PT(\D)}$, $t\bullet t_1\ldots t_k=(t\bullet t_1\ldots t_{k-1})\bullet t_k-t\bullet (t_1\ldots t_{k-1}
\bullet t_k)$.
\item If $a,b,c \in S(\g_{\PT(\D)})$, $ab \bullet c=(a\bullet c^{(1)} )(b \bullet c^{(2)})$, where $\Delta(c)=c^{(1)}
\otimes c^{(2)}$ is the usual coproduct of $S(\g_{\PT(\D)})$. In particular, if $t_1,\ldots,t_k,t\in \PT(\D)$:
$$t_1\ldots t_k \bullet t=\sum_{i=1}^k t_1\ldots (t_i \bullet t)\ldots t_k.$$
\end{enumerate}

\begin{lemma}\label{19}
Let $t=(t,I),t_1=(t_1,I^{(1)}),\ldots,t_k=(t_k,I^{(k)})$ be partitioned trees $(k\geq 1$). Let $s_1,\ldots,s_k \in V(t)$.
The partitioned tree $t\bullet_{s_1,\ldots,s_k}(t_1,\ldots,t_k)$ is obtained by grafting the roots of $t_i$ on $s_i$ for all $i$,
the partition being $I\sqcup I^{(1)}\sqcup\ldots \sqcup I^{(k)}$. Then:
$$t\bullet t_1\ldots t_k=\sum_{s_1,\ldots,s_k \in V(t)} t\bullet_{s_1,\ldots,s_k} (t_1,\ldots,t_k).$$
\end{lemma}

\begin{proof} By induction on $k$. This is obvious if $k=1$. Let us assume the result at rank $k$.
\begin{align*}
t\bullet t_1\ldots t_{k+1}&=(t\bullet t_1\ldots t_k) \bullet t_{k+1}-\sum_{i=1}^k t\bullet (t_1\ldots (t_i\bullet t_{k+1}) \ldots t_k)\\
&=\sum_{s_1,\ldots,s_k \in V(t)}(t\bullet_{s_1,\ldots,s_k} (t_1,\ldots, t_k)) \bullet t_{k+1}
-\sum_{i=1}^k \sum_{s \in V(t_i)} t\bullet (t_1\ldots (t_i \bullet_s t_{k+1})\ldots t_i)\\
&=\sum_{s_1,\ldots,s_{k+1} \in V(t)}(t\bullet_{s_1,\ldots,s_k} (t_1,\ldots,t_k) )\bullet_{s_{k+1}} t_{k+1}\\
&+\sum_{i=1}^k \sum_{s\in V(t_i)}(t\bullet_{s_1,\ldots,s_k} (t_1,\ldots, t_k)) \bullet_s t_{k+1}\\
&-\sum_{i=1}^k \sum_{s_1,\ldots,s_k \in V(t)}\sum_{s \in V(t_i)} t\bullet_{s_1,\ldots,s_k} 
(t_1,\ldots, t_i \bullet_s t_{k+1},\ldots, t_i)\\
&=\sum_{s_1,\ldots,s_{k+1} \in V(t)}t\bullet_{s_1,\ldots,s_{k+1}} (t_1,\ldots,t_{k+1}).
\end{align*}
Hence, the result holds for all $k$. \end{proof}

\begin{theo}
Let $\D$ be a set, let $A$ be a Com-Prelie algebra, and let $a_d \in A$ for all $d\in \D$.
There exists a unique morphism of Com-Prelie algebra $\phi:\g_{\PT(\D)}\longrightarrow A$, such that $\phi(\tdun{$d$})=a_d$
 for all $d\in \D$. In other words, $\g_{\PT(\D)}$ is the free Com-Prelie algebra generated by $\D$.
\end{theo}

\begin{proof} {\it Unicity.} Let $t\in \T^d$. We denote by $r_1,\ldots,r_n$ its roots. For all $1\leq i\leq n$,
let $t_{i,1},\ldots,t_{i,k_i}$ be the partitioned trees born from $r_i$ and let $d_i$ be the decoration of $r_i$. Then:
$$t=(\tdun{$d_1$}\bullet t_{1,1}\ldots t_{1,k_1} )\shuffle \ldots \shuffle (\tdun{$d_n$} \bullet t_{n,1}\ldots t_{n,k_n}).$$
So $\phi$ is inductively defined by:
\begin{equation}\label{E1}
\phi(t)=(a_{d_1}\bullet\phi(t_{1,1})\ldots \phi(t_{1,k_1}))\shuffle\ldots \shuffle(a_{d_n}\bullet
\phi(t_{n,1})\ldots \phi(t_{n,k_n})).
\end{equation}

{\it Existence.} As the product $\shuffle$ of $A$ is commutative and associative, (\ref{E1}) defines inductively 
a morphism $\phi$ from $\g_{\PT(\D)}$ to $A$. By definition, it is compatible with the product $\shuffle$. Let us prove the compatibility
with the product $\bullet$. Let $t,t'$ be two partitioned trees, let us prove that $\phi(t\bullet t')=\phi(t)\bullet \phi(t')$ by induction on
the number $N$ of vertices of $t$. If $N=1$, then $t=\tdun{$d$}$ and:
$$\phi(t\bullet t')=a_d \bullet \phi(t')=\phi(t)\bullet \phi(t'),$$
by definition of $t'$. If $N>1$, two cases are possible. 

{\it First case.} If $t$ has only one root, then $t=\tdun{$d$} \bullet t_1\ldots t_k$, and:
$$t\bullet t'=\tdun{$d$}\bullet t_1\ldots t_k t'+\sum_{i=1}^k \tdun{$d$} \bullet t_1\ldots t_i \circ t'\bullet t_k.$$
Using the induction hypothesis on $t_1,\ldots,t_k$:
\begin{align*}
\phi(t\bullet t')&=a_d \bullet \phi(t_1)\ldots \phi(t_k) \phi(t')
+\sum_{i=1}^k a_d \bullet \phi(t_1)\ldots \phi(t_1 \circ t')\ldots \phi(t_k)\\
&=a_d \bullet \phi(t_1)\ldots \phi(t_k) \phi(t')
+\sum_{i=1}^k a_d \bullet (\phi(t_1)\ldots \phi(t_1) \circ \phi(t')\ldots \phi(t_k))\\
&=(a_d \bullet \phi(t_1)\ldots \phi(t_k))\bullet \phi(t')\\
&=\phi(t)\bullet \phi(t').
\end{align*}

{\it Second case.} If $t$ has $k>1$ roots, we put $t=t_1\shuffle \ldots \shuffle t_k$. The induction hypothesis holds for $t_1,\ldots,t_k$, so:
\begin{align*}
\phi(t\circ t')&=\sum_{i=1}^k \phi(t_1\shuffle t_i \bullet t'\shuffle \ldots \shuffle t_k)\\
&=\sum_{i=1}^k \phi(t_1)\shuffle \phi(t_i \bullet t')\shuffle \ldots \shuffle \phi(t_k)\\
&=\sum_{i=1}^k \phi(t_1)\shuffle \phi(t_i) \bullet \phi(t')\shuffle \ldots \shuffle \phi(t_k)\\
&=(\phi(t_1)\shuffle \ldots \shuffle \phi(t_k))\bullet \phi(t')\\
&=\phi(t)\bullet \phi(t').
\end{align*}
Hence, $\phi$ is a morphism of Com-Prelie algebras. \end{proof}

\subsection{Presentation of $\B$ as a Com-Prelie algebra}

\begin{prop}\label{21}
As a Com-Prelie algebra, $\B$ is generated by $\emptyset$ and $x_1$.
\end{prop}

\begin{proof} Let $A$ be the Com-Prelie subalgebra of $\B$ generated by $\emptyset$ and $x_1$. For all $n \geq 1$,
it contains $x_1^{\shuffle n}=n! x_1^n$, so it contains $x_1^n$ for all $n \geq 0$. As $\B$ is generated by these elements as a prelie algebra, 
$A=\B$. \end{proof}\\

We denote by $\phi_{CPL}:\g_{\PT(\{1,2\})}\longrightarrow \B$ the unique morphism of Com-Prelie algebras which sends $\tdun{$1$}$
to $\emptyset$ and $\tdun{$2$} $ to $\tdun{$2$}$. By proposition \ref{21}, it is surjective.

\begin{lemma}\label{22}
Let $t_1,\ldots,t_k \in \PT(\{1,2\})$. 
\begin{enumerate}
\item  $\phi_{CPL}(\tdun{$1$}\bullet t_1\ldots t_k)=0$ if $k\geq 1$.
\item $\phi_{CPL}(\tdun{$2$}\bullet t_1\ldots t_k)=0$ if $k\geq 2$. 
\item If $t\in \PT(\{1,2\})$, $\phi_{CPL}(\tdun{$2$}\bullet t)=x_0\phi_{CPL}(t)$.
\end{enumerate} \end{lemma}

\begin{proof} We prove 1.-3. by induction on $k$. If $k=1$:
$$\begin{array}{rcl}
\phi_{CPL}(\tdun{$1$}\bullet t)&=\emptyset \bullet \phi_{CPL}(t)&=0,\\
\phi_{CPL}(\tdun{$2$}\bullet t)&=x_1 \bullet \phi_{CPL}(t)&=x_0 \phi_{CPL}(t).
\end{array}$$
Let us assume the results at rank $k-1\geq 1$. Then:
\begin{align*}
\phi_{CPL}(\tdun{$1$} \bullet t_1\ldots t_k)&=\emptyset \bullet \phi_{CPL}(t_1)\ldots \phi_{CPL}(t_k)\\
&=(\emptyset \bullet \phi_{CPL}(t_1)\ldots \phi_{CPL}(t_{k-1}))\bullet \phi_{CPL}(t_k)\\
&-\sum_{i=1}^k \emptyset \bullet \phi_{CPL}(t_1)\ldots \phi_{CPL}(t_i\bullet t_k)\ldots \phi_{CPL}(t_{k-1})\\
&=0,\\
\phi_{CPL}(\tdun{$2$} \bullet t_1\ldots t_k)&=x_1 \bullet \phi_{CPL}(t_1)\ldots \phi_{CPL}(t_k)\\
&=(x_1 \bullet \phi_{CPL}(t_1)\ldots \phi_{CPL}(t_{k-1}))\bullet \phi_{CPL}(t_k)\\
&-\sum_{i=1}^k x_1 \bullet \phi_{CPL}(t_1)\ldots \phi_{CPL}(t_i\bullet t_k)\ldots \phi_{CPL}(t_{k-1}).
\end{align*}
If $k\geq 3$, the induction hypothesis immediately allows to conclude that $\phi_{CPL}(\tdun{$2$} \bullet t_1\ldots t_k)=0-0=0$.
If $k=2$, this gives:
\begin{align*}
\phi_{CPL}(\tdun{$2$} \bullet t_1t_2)&=(x_1 \bullet \phi_{CPL}(t_1))\bullet \phi_{CPL}(t_2)-x_1\bullet \phi_{CPL}
(t_1\bullet t_2)\\
&=(x_0\phi_{CPL}(t_1))\bullet \phi_{CPL}(t_2)-x_0\phi_{CPL}(t_1\bullet t_2)\\
&=x_0\left(\phi_{CPL}(t_1)\bullet \phi_{CPL}(t_2))\phi_{CPL}(t_1\bullet t_2)\right)\\
&=0.
\end{align*}
Hence, the result holds for all $k\geq 1$. \end{proof}

\begin{lemma} \label{23}
For all $t\in \PT(\{1,2\})$, $\phi_{CPL}(t)$ is a linear span of words of length the number of vertices of $t$ decorated by $2$.
\end{lemma}

\begin{proof} By induction on the number of vertices $N$ of $t$. If $N=1$, then $t=\tdun{$1$}$ or $\tdun{$2$}$ and the result is obvious.
Let us assume the result at all rank $<N$. \\

{\it First case.} If $t$ has only one root, we put $t=\tdun{$i$}\bullet t_1\ldots t_k$. By the preceding lemma, we can assume that $i=2$ and $k=1$.
Then $\phi_{CPL}(t)=x_0\phi_{CPL}(t_1)$ and the result is obvious.\\

{\it Second case.} If $t$ has $k>1$ roots, we put $t=t_1\shuffle \ldots \shuffle t_k$. Then $\phi_{CPL}(t_1)$ is equal to
$\phi_{CPL}(t_1)\shuffle \ldots \shuffle \phi_{CPL}(t_k)$ and the result is immediate. \end{proof}

\begin{lemma}
We define inductively a family $F$ of elements of $\PT(\{1,2\})$ by:
\begin{enumerate}
\item $F(1)=\{\tdun{$1$},\tdun{$2$}\}$.
\item $\displaystyle F(n+1)=(\tdun{$2$}\bullet F(n))\cup \bigcup_{i=1}^{n} (F(i) \shuffle F(n+1-i))$.
\item $\displaystyle F=\bigcup_{n\geq 1} F(n)$.
\end{enumerate}
Let $t\in \PT(\{1,2\})$. Then $\phi_{CPL}(t) \neq 0$ if, and only if, $t\in F$.
\end{lemma}

\begin{proof} $\Longrightarrow$. We proceed by induction on the number $N$ of vertices of $t$. This is obvious if $N=1$.
Let us assume the result at all rank $<N$. \\

{\it First case.} If $N$ has only one root, we put $N=\tdun{$i$}\bullet t_1\ldots t_k$. By lemma \ref{22}, $i=2$ and $k=1$.
Then $\phi_{CPL}(t)=x_0 \phi_{CPL}(t_1)$. By the induction hypothesis, $t_1 \in F$, so $t\in F$.\\

{\it Second case.} If $N$ has $k>N$ roots, we put $t=t_1\shuffle \ldots \shuffle t_k$. Then:
$$\phi_{CPL}(t)=\phi_{CPL}(t_1)\shuffle \phi_{CPL}(t_2\shuffle \ldots \shuffle t_k)\neq 0,$$
so by the induction hypothesis, $t_1$ and $t_2\shuffle \ldots \shuffle t_k \in F$, and $t\in F$. \\

$\Longleftarrow$. Let $t\in T(n)$. We proceed by induction on $n$. It $n=1$, this is obvious. If $n>1$ then $t=\tdun{$2$}\bullet t'$, with $t'\in F(n-1)$, 
or $t=t'\shuffle t''$, with $t'\in F(i)$, $t''\in F(n-i)$. In the first case, by the induction hypothesis, $\phi_{CPL}(t')\neq 0$ and
$\phi_{CPL}(t)=x_0 \phi_{CPL}(t')\neq 0$. In the second case, $\phi_{CPL}(t'),\phi_{CPL}(t'')\neq 0$ by the induction hypothesis,
so $\phi_{CPL}(t)=\phi_{CPL}(t')\shuffle \phi_{CPL}(t'')\neq 0$. \end{proof}\\

{\bf Examples}.
\begin{align*}
F(1)&=\{\tdun{$1$},\tdun{$2$}\},\\
F(2)&=\{\tddeux{$2$}{$1$},\tddeux{$2$}{$2$},\hddeux{$1$}{$1$},\hddeux{$1$}{$2$},\hddeux{$2$}{$2$}\},\\
F(3)&=\left\{\tdtroisdeux{$2$}{$2$}{$1$},\tdtroisdeux{$2$}{$2$}{$2$},\hdtroisun{$2$}{$1$}{$1$},
\hdtroisun{$2$}{$2$}{$1$},\hdtroisun{$2$}{$2$}{$2$},\hdtroisdeux{$1$}{$2$}{$1$},
\hdtroisdeux{$2$}{$2$}{$1$},\hdtroisdeux{$1$}{$2$}{$2$},\hdtroisdeux{$2$}{$2$}{$2$},
\hdtroisquatre{$1$}{$1$}{$1$},\hdtroisquatre{$2$}{$1$}{$1$},\hdtroisquatre{$2$}{$2$}{$1$},\hdtroisquatre{$2$}{$2$}{$2$}\right\}.
\end{align*}

We define a second family of elements of $\PT(\{1,2\})$ in the following way:
\begin{enumerate}
\item $F'(1)=\{\tdun{$1$},\tdun{$2$}\}$.
\item $F'(2)=\{\tddeux{$2$}{$2$},\tddeux{$2$}{$1$},\hddeux{$2$}{$2$}\}$.
\item $\displaystyle F'(n+1)=(\tdun{$2$}\bullet F'(n))\cup \bigcup_{i=2}^{n-1}  \left(F'(i) \shuffle F'(n+1-i)\right)
\cup \left(\tdun{$2$}\shuffle F'(n)\right)$ if $n \geq 2$.
\item $\displaystyle F'=\bigcup_{n\geq 1} F'(n)$.
\end{enumerate}
For example:
\begin{align*}
F'(3)&=\left\{\tdtroisdeux{$2$}{$2$}{$1$},\tdtroisdeux{$2$}{$2$}{$2$},\hdtroisun{$2$}{$2$}{$2$},
\hdtroisdeux{$2$}{$2$}{$1$},\hdtroisdeux{$2$}{$2$}{$2$},\hdtroisquatre{$2$}{$2$}{$2$}\right\},\\
F'(4)&=\left\{\tdquatrecinq{$2$}{$2$}{$2$}{$1$},\tdquatrecinq{$2$}{$2$}{$2$}{$2$},
\hdquatresix{$2$}{$2$}{$2$}{$2$},\hdquatrecinq{$2$}{$2$}{$1$}{$2$},\hdquatrecinq{$2$}{$2$}{$2$}{$2$},
\hdquatretrois{$2$}{$2$}{$2$}{$2$},\hdquatredouze{$2$}{$2$}{$2$}{$2$},
\hdquatredix{$2$}{$2$}{$2$}{$1$},\hdquatredix{$2$}{$2$}{$2$}{$2$},
\hdquatretreize{$2$}{$1$}{$2$}{$1$},\hdquatretreize{$2$}{$1$}{$2$}{$2$},\hdquatretreize{$2$}{$2$}{$2$}{$2$},
\hdquatreseize{$2$}{$2$}{$2$}{$1$},\hdquatreseize{$2$}{$2$}{$2$}{$2$},
\hdquatredixsept{$2$}{$2$}{$2$}{$2$}\right\}.
\end{align*}
We define a map $\pi$ from $F$ to $\PT(\{1,2\})$ in the following way:
\begin{enumerate}
\item $\pi(\tdun{$i$})=\tdun{$i$}$ if $i=1,2$.
\item $\pi(\tdun{$1$}\shuffle \ldots \shuffle\tdun{$1$})=\tdun{$1$}$.
\item If $t=\tdun{$1$}\shuffle \ldots \shuffle \tdun{$1$} \shuffle t_1\shuffle \ldots \shuffle t_k$, $k \geq 1$, with $t_1,\ldots,t_k \neq \tdun{$1$}$,
then $\pi(t)=\pi(t_1)\shuffle \ldots \shuffle \pi(t_k)$.
\item If $t=\tdun{$2$} \bullet t_1\ldots t_k$, then $\pi(t)=\tdun{$2$}\bullet \pi(t_1)\ldots \pi(t_k)$.
\end{enumerate}

\begin{lemma}
$\pi$ is a projection on $F'$ and $\phi_{CPL}\circ \pi={\phi_{CPL}}_{\mid F}$.
\end{lemma}

\begin{proof} Let $t\in F$. Let us prove by induction on the number $N$ of vertices of $t$ that:
\begin{enumerate}
\item $\pi(t)\in F'$.
\item If $t\in F'$, $\pi(t)=t$.
\item $\phi_{CPL} \circ \pi(t)=\phi_{CPL}(t)$.
\item If $\pi(t)=\tdun{$1$}$, then $t=\tdun{$1$}^{\shuffle N}$.
\end{enumerate}
All these points are immediate if $N=1$. Let us assume the result at all ranks $<N$, $N\geq 2$.
We put  $t=\tdun{$1$}\shuffle \ldots \shuffle \tdun{$1$} \shuffle t_1\shuffle \ldots \shuffle t_k$, $k \geq 0$, with $t_1,\ldots,t_k \neq \tdun{$1$}$.\\

{\it First case.} If $k \geq 2$, then $\pi(t)=\pi(t_1)\shuffle \ldots \shuffle \pi(t_k)$. Following the induction hypothesis, 
$\pi(t_1),\ldots,\pi(t_k)\in F'$ and are not equal to $\tdun{$1$}$, so $\pi(t)\in F'$: moreover, $\pi(t_1)\neq \tdun{$1$}$, so $\pi(t)
\neq \tdun{$1$}$.
\begin{align*}
\phi_{CPL}(t)&=\phi_{CPL}(\tdun{$1$})\shuffle\ldots \shuffle \phi_{CPL}(\tdun{$1$})\shuffle
\phi_{CPL}(t_1)\shuffle \ldots \shuffle \phi_{CPL}(t_k)\\
&=\emptyset \shuffle\ldots \shuffle \emptyset \shuffle \phi_{CPL}\circ \pi(t_1)\shuffle \ldots \shuffle \phi_{CPL}\circ \pi(t_k)\\
&=\phi_{CPL}(\pi(t_1)\shuffle \ldots \shuffle \pi(t_k))\\
&=\phi_{CPL}\circ \pi(t).
\end{align*}
If $t\in F'$, necessarily $t=t_1\shuffle \ldots \shuffle t_k$, and $t_1,\ldots,t_k \in F'$. By the induction hypothesis,
$\pi(t_1)=t_1,\ldots,\pi(t_k)=t_k$, so $\pi(t)=t$.  \\

{\it Second case.} If $k=1$, as $t_1\in F$, we put $t_1=\tdun{$2$}\bullet s$. Then $\pi(t)=\tdun{$2$}\bullet \pi(s)$.
By the induction hypothesis, $\pi(s) \in F'$, so $\pi(t)=F'$. Moreover:
\begin{align*}
\phi_{CPL}(t)&=\phi_{CPL}(\tdun{$1$})\shuffle\ldots \shuffle \phi_{CPL}(\tdun{$1$})\shuffle
(\phi_{CPL}(\tdun{$2$})\bullet \phi_{CPL}(s))\\
&=\emptyset\shuffle\ldots \shuffle \emptyset \shuffle (\phi_{CPL}(\tdun{$2$})\bullet \phi_{CPL}(s))\\
&=\phi_{CPL}\circ \pi(\tdun{$2$})\bullet \phi_{CPL}\circ \pi(s)\\
&=\phi_{CPL}\circ \pi(t).
\end{align*}
If $t'\in F'$, then $s \in F'$, and $t=\tdun{$2$}\bullet s$. Then $\pi(t)=\tdun{$2$}\bullet \pi(s)=\tdun{$2$}\bullet s=t$. \\

{\it Last case.} If $k=0$, all the results are obvious. \end{proof}

\begin{lemma}\label{26}
Let $t,t'\in \PT(\{1,2\})$. Then:
$$\phi_{CPL}\left((\tdun{$2$} \bullet t) \shuffle (\tdun{$2$} \bullet t')\right)=
\phi_{CPL}\left(\tdun{$2$}\bullet((\tdun{$2$}\bullet t) \shuffle t'+t\shuffle (\tdun{$2$}\bullet t'))\right).$$
\end{lemma}

\begin{proof} Indeed, putting $w=\phi_{CPL}(t)$ and $w'=\phi_{CPL}(t')$:
\begin{align*}
\phi_{CPL}\left((\tdun{$2$} \bullet t) \shuffle (\tdun{$2$} \bullet t')\right)&=x_0w \shuffle x_0w'\\
&=x_0 (w \shuffle x_0w')+x_0(x_0w \shuffle w')\\
&=\phi_{CPL}\left(\tdun{$2$}\bullet((\tdun{$2$}\bullet t) \shuffle t'+t\shuffle (\tdun{$2$}\bullet t'))\right).
\end{align*}
We used lemma \ref{22} for the first and third equalities. \end{proof}

\begin{theo}\label{27}
The kernel of $\phi_{CPL}$ is the Com-Prelie ideal generated by the elements:
\begin{enumerate}
\item $\tdun{$1$}\bullet t_1\ldots t_k$, where $k\geq 1$, $t_1,\ldots,t_k \in \PT(\{1,2\})$.
\item $\tdun{$2$}\bullet t_1\ldots t_k$, where $k\geq 2$, $t_1,\ldots,t_k \in \PT(\{1,2\})$.
\item $\tdun{$1$} \shuffle t-t$, where $t\in \PT(\{1,2\})$.
\item $(\tdun{$2$} \bullet t) \shuffle (\tdun{$2$} \bullet t')
-\tdun{$2$}\bullet((\tdun{$2$}\bullet t) \shuffle t'-t\shuffle (\tdun{$2$}\bullet t'))$, where $t,t'\in \PT(\{1,2\})$.
\end{enumerate}\end{theo}

\begin{proof} Let $I$ be the ideal generated by these elements. Lemmas \ref{22} and \ref{26}  prove that the elements 1, 2  and 4 belong to 
$Ker(\phi_{CPL})$. Moreover, for all $t \in \PT(\{1,2\})$, $\pi(\tdun{$1$}\shuffle t)=\pi(t)$. For all $t\in \PT(\{1,2\})$:
$$\phi_{CPL}(\tdun{$1$}\shuffle t)=\emptyset \shuffle \phi_{CPL}(t)=\phi_{CPL}(t),$$
so elements 3. also belong to $Ker(\phi_{CPL})$. Hence, $I\subseteq Ker(\phi_{CPL})$. \\

Let $h=\g_{\PT(\{1,2\})}/I$. As the elements 1 and 2 belong to $I$, $h$ is linearly spanned by the elements $\overline{t}$, $t\in F$.
As the elements 3 belong to $I$, for all $t\in F$, $\overline{\pi(t)}=\overline{t}$. As $\pi$ is a projection on $F'$, $h$ is linearly spanned 
by the elements $\overline{t}$, $t\in F'$. \\

We now define inductively two families of partitionned trees in the following way:
\begin{enumerate}
\item $T''(1)=\{\tdun{$2$}\}$ and $F''(1)=\{\tdun{$1$},\tdun{$2$}\}$.
\item $T''(n+1)=\tdun{$2$}\bullet F''(n)$.
\item $\displaystyle F''(n+1)=\bigcup_{i=1}^{n+1} T''(i)\shuffle \tdun{$2$}^{\shuffle(n+1-i)}$.
\item $\displaystyle F''=\bigcup_{n\geq 1} F''(n)$.
\end{enumerate}
For example:
\begin{align*}
F''(3)&=\left\{\tdtroisdeux{$2$}{$2$}{$1$},\tdtroisdeux{$2$}{$2$}{$2$},\hdtroisun{$2$}{$2$}{$2$},
\hdtroisdeux{$2$}{$2$}{$1$},\hdtroisdeux{$2$}{$2$}{$2$},\hdtroisquatre{$2$}{$2$}{$2$}\right\},\\
F''(4)&=\left\{\tdquatrecinq{$2$}{$2$}{$2$}{$1$},\tdquatrecinq{$2$}{$2$}{$2$}{$2$},
\hdquatresix{$2$}{$2$}{$2$}{$2$},\hdquatrecinq{$2$}{$2$}{$1$}{$2$},\hdquatrecinq{$2$}{$2$}{$2$}{$2$},
\hdquatretrois{$2$}{$2$}{$2$}{$2$},\hdquatredouze{$2$}{$2$}{$2$}{$2$},
\hdquatredix{$2$}{$2$}{$2$}{$1$},\hdquatredix{$2$}{$2$}{$2$}{$2$},
\hdquatreseize{$2$}{$2$}{$2$}{$1$},\hdquatreseize{$2$}{$2$}{$2$}{$2$},
\hdquatredixsept{$2$}{$2$}{$2$}{$2$}\right\}.
\end{align*}
Let us prove that for all $t\in F'$, there exists $t' \in Vect(F'')$ such that $\overline{t}=\overline{t'}$. We proceed by induction on the number 
$N$ of vertices of $t$. If $N=1$, then $t=\tdun{$1$} $ or $\tdun{$2$}$ and we take $t'=t$. Let us assume the result at all rank $<N$. 
We put $t=t_1\shuffle \ldots \shuffle t_k \shuffle \tdun{$2$}\shuffle \ldots \shuffle \tdun{$2$}$, with $t_i=\tdun{$2$}\bullet s_i$, $s_i \neq 1$, 
for all $1\leq i \leq k$. We proceed by induction on $k$. If $k=0$, we take $t'=t=\tdun{$2$}\shuffle \ldots \shuffle \tdun{$2$}$. If $k=1$, then,
by the induction hypothesis on $N$ applied to $s_1$:
$$\overline{t}=(\overline{\tdun{$2$}} \bullet \overline{s_1})\shuffle \overline{\tdun{$2$}}\shuffle \ldots \shuffle \overline{\tdun{$2$}}
=(\overline{\tdun{$2$}} \bullet \overline{s_1'})\shuffle \overline{\tdun{$2$}}\shuffle \ldots \shuffle \overline{\tdun{$2$}}
=\overline{(\tdun{$2$} \bullet s_1')\shuffle \tdun{$2$}\shuffle \ldots \shuffle \tdun{$2$}}.$$
We take $t'=(\tdun{$2$} \bullet s_1')\shuffle \tdun{$2$}\shuffle \ldots \shuffle \tdun{$2$}$, which clearly belongs to $Vect(F'')$, 
as $s_1' \in Vect(F'')$. Let us assume the result at all rank $<k$. Then, as the elements 4 belong to $I$:
$$\overline{t_1\shuffle t_2}=\overline{\underbrace{\tdun{$2$}\bullet (t_1\shuffle s_2)}_{t'_1}}
+\overline{\underbrace{\tdun{$2$}\bullet(s_1 \bullet t_2)}_{t''_1}},$$
so:
$$\overline{t}=\overline{t'_1\shuffle t_3\shuffle \ldots \shuffle t_k\shuffle  \tdun{$2$}\shuffle \ldots \shuffle \tdun{$2$}}
+\overline{t''_1\shuffle t_3\shuffle \ldots \shuffle t_k\shuffle  \tdun{$2$}\shuffle \ldots \shuffle \tdun{$2$}}.$$
By the induction hypothesis on $k$ applied to these two partitionned trees, there exists $x'_1$ and $x''_1 \in Vect(F''),$ such that $\overline{t}=
\overline{x'_1}+\overline{x''_1}$. We take $t'=x_1'+x_1''$. Consequently, the elements $\overline{t}$, $t\in F''$, linearly span $h$.  \\

Let $t\in F''(n)$. Then it has $n$ vertices, and at most one of them is decorated by $1$. We denote by $F''_1(n)$ 
the set of elements of $F''(n)$ with one vertex decorated by $1$, and we put $F''_2(n)=F''(n)\setminus F''_1(n)$. 
Let us prove that for all $n \geq 1$, $|F''_1(n+1)| \leq 2^{n-1}$  and $|F''_2(n)| \leq 2^{n-1}$. For $n=0$, as 
$FF'_1(2)=\{\tddeux{$2$}{$1$}\}$ and $F''_2(1)=\{\tdun{$2$}\}$, this is immediate. Let us assume the result at all rank $\leq n$. Then:
$$F''_2(n+1)=\bigcup_{i=1}^{n+1} \tdun{$2$}^{\shuffle (n+1-i)} \shuffle T''(i)\cap F''_2(i)
=\{\tdun{$2$}^{\shuffle (n+1)}\} \cup \bigcup_{i=1}^n 
 \tdun{$2$}^{\shuffle (n+1-i)} \shuffle \tdun{$2$}\bullet F''_2(i).$$
Hence, $|F''_2(n+1)|\leq 1+1+2+\ldots+2^{n-1}=2^n$. 
$$F''_1(n+2)=\bigcup_{i=1}^{n+2} \tdun{$2$}^{\shuffle(n+2-i)} \shuffle T''(i)\cap F''_1(i)
=\bigcup_{i=2}^{n+2} \tdun{$2$}^{\shuffle(n+2-i)} \shuffle \tdun{$2$}\bullet F''_1(i-1).$$
Hence, $|F''_1(n+2)|\leq +1+1+\ldots+2^{n-1}=2^n$. \\

Let $\overline{\phi}_{APL}$ be the linear map induced by $\phi_{CPL}$ on $h$. 
If $t \in F''_1(n)$, by lemma \ref{23}, $\overline{\phi}_{APL}(\overline{t})$ is a linear span of words of length $n-1$.
If $t \in F''_2(n)$, by lemma \ref{23}, $\overline{\phi}_{APL}(\overline{t})$ is a linear span of words of length $n$. Hence, for all $n \geq 0$:
$$\overline{\phi}_{APL}(Vect(F''_2(n))+Vect(F''_1(n+1)))\subseteq Vect(\mbox{words of length $n$}).$$
As $\phi_{CPL}$ is surjective, we obtain:
$$\overline{\phi}_{APL}(Vect(F''_2(n))+Vect(F''_1(n+1)))=Vect(\mbox{words of length $n$}).$$
Moreover, as $dim(Vect(\mbox{words of length $n$}))=2^n$ and 
$dim(Vect(F''_2(n))+Vect(F''_1(n+1)))\leq |F''_2(n)|+|F''_1(n)|\leq 2^{n-1}+2^{n-1}=2^n$, 
the restriction of $\overline{\phi}_{APL}$ to $Vect(F''_2(n))+Vect(F''_1(n+1))$ is injective. 
Finally, $\overline{\phi}_{APL}$ is injective, so $Ker(\phi_{CPL})=I$. \end{proof}

\section{Presentation of $\B$ as a prelie algebra}

\subsection{A surjective morphism}

Let $\g_{\T(\N^*)}$ be the free prelie algebra generated by $\N^*$, as described in \cite{Chapoton}. 
It can be seen as the subspace of $\g_{\PT(\N^*)}$ generated by rooted trees (which are seen as partitioned trees such that any part of 
the partition is a singleton), with the restriction of the prelie product $\bullet$ defined by graftings. For example, in $\g_{\T(\N^*)}$, if $a,b,c,d>0$:
$$\tddeux{$a$}{$b$}\bullet \tddeux{$c$}{$d$}=\tdquatretrois{$a$}{$c$}{$d$}{$b$}+\tdquatrecinq{$a$}{$b$}{$c$}{$d$}.$$
This prelie algebra is graded, the degree of a tree being the sum of its decorations. \\

By theorem \ref{12}, there exists a unique surjective map of prelie algebras $\Phi_{PL}:\g_{\T(\N^*)}\longrightarrow \B$, 
sending $\tdun{$n$}$ to $x_1^{n-1}$ for all $n\geq 1$. As $x_1^{i-1}$ is homogeneous of degree $i$ for all $i$, 
this morphism is homogeneous of degree $0$.\\

{\bf Notation.} If $t_1\ldots t_k\in \T(\N^*)$ and $n\in \N^*$, we put:
$$B_n(t_1\ldots t_k)=\tdun{$n$}\bullet t_1\ldots t_k.$$
This is the tree obtained by grafting $t_1,\ldots,t_k$ on a common root decorated by $n$.

\begin{prop}\label{28}
Let $t=B_n( t_1\ldots t_k )\in \T(\N^*)$. We put $\phi_{PL}(t_i)=w_i$ for all $1\leq i \leq k$. Then:
$$\phi_{PL}(t)=\begin{cases}
x_0w_1\shuffle \ldots \shuffle x_0w_k \shuffle x_1^{n-1-k} \mbox{ if }k<n,\\
0\mbox{ otherwise}. 
\end{cases}$$ \end{prop}

\begin{proof} As  $\g_{\PT(\{1,2\})}$ is prelie, there exists a unique morphism of prelie algebras:
$$\psi:\begin{cases}
\g_{\T(\N^*)}&\longrightarrow\g_{\PT(\{1,2\})}\\
\tdun{$n$}&\longrightarrow\frac{1}{(n-1)!} \tdun{$2$}^{\shuffle (n-1)}.
\end{cases}$$
Then $\phi_{APL}\circ \psi$ is a prelie algebra morphism sending $\tdun{$n$} $ to $\frac{1}{(n-1)!} x_1^{\shuffle (n-1)}
=x_1^{n-1}$ for all $n \geq 1$, so $\phi_{APL}\circ \psi=\phi_{PL}$. We obtain, by lemma \ref{19}:
\begin{align*}
\psi(\tdun{$n$} \bullet t_1\ldots t_k)&=\frac{1}{(n-1)!} \tdun{$2$}^{\shuffle(n-1)}\bullet (\psi(t_1)\ldots \psi(t_k))\\
&=\frac{1}{(n-1)!}\sum_{I_1\sqcup \ldots \sqcup I_{n-1}=\{1,\ldots,k\}}
\tdun{$2$}\bullet \left(\prod_{i\in I_1} t_i\right)\shuffle \ldots \shuffle \tdun{$2$}\bullet \left(\prod_{i\in I_{n-1}} t_i\right)
\end{align*}
Let us apply $\phi_{APL}$ to this expression. If $|I_j|\geq 2$, by theorem \ref{27}:
$$\phi_{APL}\left( \tdun{$2$}\bullet \left(\prod_{i\in I_j} t_i\right)\right)=0.$$
Consequently, if $k\geq n$, at least one of the $I_j$ contains two elements, so $\phi_{APL}\circ \psi(t)=\phi_{PL}(t)=0$.
Let us assume that $k<n$. Hence, using the commutativity of $\shuffle$:
\begin{align*}
\phi_{PL}(\tdun{$n$} \bullet t_1\ldots t_k)&=\frac{1}{(n-1)!}\sum_{I_1\sqcup \ldots \sqcup I_{n-1}=\{1,\ldots,k\},\: |I_j|\leq 1}
x_1 \bullet \left(\prod_{i\in I_1} w_i\right)\shuffle \ldots \shuffle x_1\bullet \left(\prod_{i\in I_k} w_i\right)\\
&=\frac{1}{(n-1)!}\sum_{\iota:\{1,\ldots,k\}\longrightarrow \{1,\ldots,n-1\},\mbox{\scriptsize{ injective}}}
x_1 \bullet w_1\shuffle \ldots x_1\bullet w_k \shuffle x_1^{\shuffle (n-1-k)}\\
&=\frac{1}{(n-1)!}\sum_{\iota:\{1,\ldots,k\}\longrightarrow \{1,\ldots,n-1\},\mbox{\scriptsize{ injective}}}
x_0 w_1\shuffle \ldots x_0 w_k \shuffle x_1^{\shuffle (n-1-k)}\\
&=\frac{(n-1)\ldots (n-k)}{(n-1)!}x_0 w_1\shuffle \ldots x_0 w_k \shuffle x_1^{\shuffle (n-1-k)}\\
&=\frac{(n-1)\ldots (n-k)(n-1-k)! }{(n-1)!}x_0 w_1\shuffle \ldots x_0 w_k \shuffle x_1^{n-1-k}\\
&=x_0 w_1\shuffle \ldots x_0 w_k \shuffle x_1^{n-1-k},
\end{align*}
which is the announced result. \end{proof}

\begin{cor} \label{29}
Let $s_1,\ldots,s_k,t_1,\ldots,t_l \in \T(\{N^*)$, $k,l\geq 0$. For all $i,j,n \geq 1$:
\begin{align*}
&\phi_{PL}\left(B_{n+1}((B_i(s_1\ldots s_k) B_j(t_1\ldots t_l))\right)\\
&=\phi_{PL}\left(B_n(B_{i+1}(s_1\ldots s_k B_j(t_1\ldots t_l))\right)
+\phi_{PL}\left(B_n(B_{j+1}(B_i(s_1\ldots s_k)t_1\ldots t_l)\right).
\end{align*}\end{cor}

\begin{proof} We note:
$$\begin{array}{rcccl}
T_1&=B_{n+1}((B_i(s_1\ldots s_k) B_j(t_1\ldots t_l))
&=\tdun{$n+1$}\hspace{.4cm}\bullet ((\tdun{$i$}\bullet s_1\ldots s_k) (\tdun{$j$} \bullet t_1\ldots t_l)),\\
T_2&=B_n(B_{i+1}(s_1\ldots s_k B_j(t_1\ldots t_l))
&=\tdun{$n$} \bullet (\tdun{$i+1$}\hspace{.3cm}\bullet (s_1\ldots s_k( \tdun{$j$}\bullet t_1\ldots t_l))),\\
T_3&=B_n(B_{j+1}(B_i(s_1\ldots s_k)t_1\ldots t_l)
&=\tdun{$n$} \bullet(\tdun{$j+1$}\hspace{.3cm}\bullet((\tdun{$i$}\bullet s_1\ldots s_k) t_1\ldots t_l)).
\end{array}$$
If $k\geq i$, or $l\geq j$, or $n=1$, all these elements are sent to zero by $\phi_{PL}$ by proposition \ref{28}. Let us assume now that $k<i$,
$l<j$, $n<1$. We put $v_i=\phi_{PL}(s_i)$ and $w_i=\phi_{PL}(t_i)$. Then:
\begin{align*}
\phi_{PL}(T_1)&=x_0(\underbrace{x_0v_1\shuffle \ldots \shuffle x_0v_k \shuffle x_1^{i-1-k})}_{X}
\shuffle x_0(\underbrace{x_0w_1\shuffle\ldots \shuffle x_0w_l \shuffle x_1^{j-1_l})}_{Y}\shuffle x_1^{n-2}\\
&=x_0X \shuffle x_0Y\shuffle x_1^{n-2},\\
\phi_{PL}(T_2)&=x_0(x_0v_1\shuffle \ldots \shuffle x_0(x_0w_1\shuffle \ldots \shuffle x_0w_l \shuffle x_1^{j-1-l})
\shuffle x_1^{i-1-k}) \shuffle x_1^{n-2}\\
&=x_0(X \shuffle x_0Y) \shuffle x_1^{n-2},\\
\phi_{PL}(T_3)&=x_0(x_0(x_0v_1\shuffle \ldots \shuffle x_0v_k \shuffle x_1^{i-1-k})\shuffle x_0w_1
\shuffle x_0w_l\shuffle x_1^{j-1-l}) \shuffle x_1^{n-2}\\
&=x_0(x_0X \shuffle Y) \shuffle x_1^{n-2}.
\end{align*}
As $x_0X \shuffle x_0Y=x_0(X \shuffle x_0Y)+x_0(x_0X\shuffle Y)$, we obtain the result. \end{proof}

\begin{theo}\label{30}
The kernel of $\phi_{PL}$ is the prelie ideal generated by:
\begin{enumerate}
\item $B_1(t_1\ldots t_k)$, where $k\geq 1$, $t_1,\ldots, t_k \in \T(\N^*)$.
\item $B_{n+1}(B_i(s_1\ldots s_k)B_j(t_1\ldots t_l))-B_n(B_{i+1}(s_1\ldots s_kB_j(t_1\ldots t_l))-
B_{j+1}(B_i(s_1\ldots s_k) t_1\ldots t_l))$, where $k,l\geq 0$, $s_1,\ldots,s_k,t_1,\ldots,t_l \in \T(\N^*)$.
\end{enumerate}\end{theo}

\begin{proof} Let $I$ be the ideal generated by these elements. By proposition \ref{28} and corollary \ref{29}, $I\subseteq Ker(\phi_{PL})$.
We put $h=\g_{\T(\N^*)}/I$. Applying repeatedly the relation given by elements of the second form, it is not difficult to prove that for any 
$t \in \T(\N^*)$, there exists a linear span of ladders $t'$ such that $\overline{t}=\overline{t'}$ in $h$. Moreover, by the relation 
given by elements 1., if one of the vertices of a ladder $t$ which is not the leaf is decorated by $1$, then $\overline{t}=0$.
Let us denote by $L(n)$ the set of ladders decorated by $\N^*$, of weight $n$, such that all the vertices which are not the leaf are decorated by 
integerS $>1$. It turns out that $h$ is generated by the elements $\overline{t}$, $t\in L=\bigcup L(n)$. 

Let $\overline{\phi_{PL}}$ be the morphism form $h$ to $\B$ induced by $\phi_{PL}$. By homogeneity, as $\phi_{PL}$ is surjective, 
for all $n \geq 1$:
$$\overline{\phi}_{PL}(Vect(L(n)))=Vect(\mbox{words of degree }n).$$
In order to prove that $I=Ker(\phi_{PL})$, it is enough to prove that $\overline{\phi}_{PL}$ is injective. By homogeneity, it is enough to prove
that $\overline{\phi}_{\mid Vect(L(n))}$ is injective for all $n\geq 1$. Hence, it is enough to prove that for all $n \geq 1$,
$$|L(n)| = dim(Vect(\mbox{words of degree }n))=p_n,$$
where the $p_n$ are the integers defined in proposition \ref{8}. Let $l_n=|L(n)|$ and $q_n$ be the number of $t\in L(n)$ with no vertex
decorated by $1$.  Then for all $n \geq 2$, $l_n=q_n+q_{n-1}$, and $l_1=1$. We put:
$$L=\sum_{n=1}^\infty l_n X^n,\: Q=\sum_{n=1}^\infty q_nX^n.$$
We obtain $P=X+Q+XQ$. Moreover:
$$Q=\frac{1}{\displaystyle 1-\sum_{i\geq 2}X^i}-1=\frac{1}{1-\frac{X^2}{1-X}}-1=\frac{X^2}{1-X-X^2},$$
Finally:
$$L=\frac{X}{1-X-X^2}=F.$$
So, for all $n \geq 1$, $|L(n)|=p_n$. \end{proof}\\

As an immediate corollary, a basis of $h$ is given by the classes of the elements of $L$. Turning to $\B$, we obtain:

\begin{cor}\label{31}
Let $w=a_1\ldots a_k$ be a word with letters in $\N^*$. 
\begin{enumerate}
\item We put:
$$m_w=x_1^{a_1-1}\bullet(x_1^{a_1-1}\bullet(\ldots(x_1^{a_{k-1}-1}\bullet x_1^{a_k})\ldots).$$
\item We shall say that $w$ is \emph{admissible} if $a_1,\ldots,a_{k-1}>1$. The set of admissible words is denoted by $\Adm$.
\end{enumerate}
Then $(m_w)_{w \in \Adm}$ is a basis of $\B$.
\end{cor}

{\bf Remark.} If $w$ is not admissible, that is to say if there exists $1\leq i<k$, such that $a_i=1$, then $m_w=0$ by proposition \ref{28}. \\

We extend the map $w\longrightarrow m_w$ by linearity.

\subsection{Prelie product in the basis of admissible words}

{\bf Notations.}
\begin{enumerate}
\item  For all $k,l$, we denote by $Sh(k,l)$ the set of $(k,l)$- shuffles, that is to say permutations $\zeta \in \mathfrak{S}_{k+l}$ such that
$\zeta(1)<\ldots<\zeta(k)$, $\zeta(k+1)<\ldots<\zeta(k+l)$.
\item For all $k,l$ we denote by $Sh_\prec(k,l)$ the set of $(k,l)$-shuffles $\zeta$ such that $\zeta^{-1}(k+l)=k$.
\item For all $k,l$ we denote by $Sh_\succ(k,l)$ the set of $(k,l)$-shuffles $\zeta$ such that $\zeta^{-1}(k+l)=k+l$.
\item The symmetric group $\mathfrak{S}_n$ acts on the set of words with letters in $\N^*$ of length $n$ by permutation of the letters:
$$\sigma.(a_1\ldots a_n)=a_{\sigma^{-1}(1)}\ldots a_{\sigma^{-1}(n)}.$$
\end{enumerate}

\begin{prop}\label{32}
Let $\K \langle\N^*\rangle$ be the space generated by words with letters in $\N^*$. We define a dendriform structure on this space by:
\begin{align*}
(a_1\ldots a_k)\prec(b_1\ldots b_l)&=\sum_{\zeta \in Sh_\prec(k,l)}\zeta.a_1\ldots a_k b_1\ldots b_{k-1}(b_k+1)\\
(a_1\ldots a_k)\succ(b_1\ldots b_l)&=\sum_{\zeta \in Sh_\succ(k,l)}\zeta.a_1\ldots a_{k-1}(a_k+1) b_1\ldots b_k.
\end{align*}
The associative product $\prec+\succ$ is denoted by $\star$.
\end{prop}

\begin{proof} We denote by $Sh(k,l,m)$ the set of $k+l+m$-permutations such that $\zeta(1)<\ldots <\zeta(k)$,
$\zeta(k+1)<\ldots<\zeta(k+l)$, $\zeta(k+l+1)<\ldots \zeta(k+l+m)$. Then:
\begin{align*}
&(a_1\ldots a_k \prec b_1\ldots b_l)\prec c_1\ldots c_m=a_1\ldots a_k \prec (b_1\ldots b_l \star c_1\ldots c_m)\\
&=\sum_{\zeta \in Sh(k,l,m), \zeta^{-1}(k+l+m)=k}\zeta.a_1\ldots a_k b_1\ldots (b_l+1)c_1\ldots (c_m+1);\\
&(a_1\ldots a_k \succ b_1\ldots b_l)\prec c_1\ldots c_m=a_1\ldots a_k \succ (b_1\ldots b_l \prec c_1\ldots c_m)\\
&=\sum_{\zeta \in Sh(k,l,m), \zeta^{-1}(k+l+m)=k+l}\zeta.a_1\ldots (a_k+1) b_1\ldots b_l c_1\ldots (c_m+1);\\
&(a_1\ldots a_k \star b_1\ldots b_l)\succ c_1\ldots c_m=a_1\ldots a_k \succ (b_1\ldots b_l \succ c_1\ldots c_m)\\
&=\sum_{\zeta \in Sh(k,l,m), \zeta^{-1}(k+l+m)=k+l+m}\zeta.a_1\ldots (a_k+1) b_1\ldots (b_l+1)c_1\ldots c_m.
\end{align*}
So $\K\langle \langle\N^*\rangle\rangle$ is a dendriform algebra. \end{proof}\\

We postpone the study of this dendriform algebra to section \ref{s52}.\\

{\bf Notations.} For all $a_1,\ldots,a_k \in \N^*$, we denote by $l(a_1\ldots a_k)=B_{a_1}\circ \ldots \circ B_{a_k}(1)$ 
the ladder decorated from the root to the leaf by $a_1,\ldots,a_k$. Note that $m_{a_1\ldots a_k}=\phi_{PL}(l(a_1\ldots a_k))$.

\begin{lemma}
Let $k,l\geq 1$ and let $a_1,\ldots,a_l,b_1,\ldots,b_l\in \N^*$. Then:
$$\phi_{PL}(B_{a_1+1}(l(a_2\ldots a_k)l(b_1\ldots b_l))+B_{b_1+1}(l(a_1\ldots a_k)l(b_2\ldots b_l))
=m_{a_1\ldots a_k \star b_1\ldots b_l}.$$
\end{lemma}

\begin{proof} By induction on $k+l$. If $k=l=1$, then:
$$\phi_{PL}(\tddeux{$a_1+1$}{$b_1$}\hspace{.5cm}+\tddeux{$b_1+1$}{$a_1$}\hspace{.5cm})
=m_{(a_1+1)b_1+(b_1+1)a_1}=m_{a_1 \star b_1}.$$
Let us assume the result at all ranks $<k+l$. If $k=1$, then:
\begin{align*}
&\phi_{PL}(B_{a_1+1}(l(b_2\ldots b_l))+B_{b_1+1}(l(a_1)l(b_2\ldots b_l))\\
&=\phi_{PL}(\tdun{$a_1+1$}\hspace{.5cm} \bullet l(b_2\ldots b_l)+\tdun{$b_1+1$}\hspace{.5cm}\bullet (l(a_1)
l(b_2\ldots b_l)))\\
&=\phi_{PL}(l((a_1+1)b_2\ldots b_l))+\phi_{PL}(\tdun{$b_1$}\bullet (l((a_1+1)b_2\ldots b_l)+\tdun{$b_2+1$}
\hspace{.5cm}\bullet(l(a_1)l(b_3\ldots b_l)))\\
&=m_{(a_1+1)b_2\ldots b_l}+m_{b_1(a_1\star b_2\ldots b_l)}\\
&=m_{(a_1+1)b_2\ldots b_l}+\sum_{i=1}^{l-1}m_{b_1\ldots b_i (a_1+1)\ldots b_l}+m_{b_1\ldots (b_l+1)a_1}\\
&=m_{a_1\star b_1\ldots b_l}.
\end{align*}
If $l=1$, a similar computation, permuting the $a_i$'s and the $b_j$'s, proves the result. If $k,l>1$, then:
\begin{align*}
&\phi_{PL}(B_{a_1+1}(l(a_2\ldots a_k)l(b_1\ldots b_l))+B_{b_1+1}(l(a_1\ldots a_k)l(b_2\ldots b_l))\\
&=\phi_{PL}(\tdun{$a_1$}\bullet(\tdun{$a_2+1$}\hspace{.5cm} \bullet l(a_3\ldots a_k)l(b_1\ldots b_l))+
\tdun{$b_1+1$} \hspace{.5cm} \bullet l(a_1\ldots a_k)l(b_2\ldots b_l)))\\
&+\phi_{PL}(\tdun{$b_1$}\bullet(\tdun{$a_1+1$}\hspace{.5cm} \bullet l(a_2\ldots a_k)l(b_2\ldots b_l))+
\tdun{$b_2+1$} \hspace{.5cm} \bullet l(a_1\ldots a_k)l(b_3\ldots b_l)))\\
&=m_{a_1(a_2\ldots a_k \star b_1\ldots b_l)+b_1(a_1\ldots a_k\star b_2\ldots b_l)}\\
&=m_{a_1\ldots a_k \star b_1\ldots b_l}.
\end{align*}
Hence, the result holds for all $k,l\geq 1$. \end{proof}

\begin{theo}\label{34}
For all $a_1,\ldots,a_k,b_1,\ldots,b_l \in \N^*$:
$$m_{a_1\ldots a_k} \bullet m_{b_1\ldots b_l}=\displaystyle \sum_{i=1}^{k-1} 
m_{a_1\ldots a_{i-1}(a_i-1)(a_{i+1}\ldots a_k \star b_1\ldots b_l)}+m_{a_1\ldots a_kb_1\ldots b_l}.$$
\end{theo}

\begin{proof} By definition of $m_{a_1 b_1\ldots b_l}$, if $k=1$, $m_{a_1}\bullet m_{b_1\ldots b_l}=m_{a_1b_1\ldots b_l}$.
So the result holds if $k=1$. Let us assume that $k\geq 2$.  In $\g_{\T(\N^*)}$, we have:
$$l(a_1\ldots a_k) \bullet l(b_1\ldots b_l)
=\tdun{$a_1$}\bullet (l(a_2\ldots a_k)\bullet l(b_1\ldots b_l))+\tdun{$a_1$}\bullet l(a_2\ldots a_k)l(b_1\ldots b_l).$$
Applying $\phi_{PL}$:
\begin{align*}
m_{a_1\ldots a_k}\bullet m_{b_1\ldots b_l}&=m_{a_1 (a_2\ldots a_k)\bullet (b_1\ldots b_l)}\\
&+\phi_{PL}(\tdun{$a_1-1$}\hspace{.5cm}\bullet(\tdun{$a_2+1$}\hspace{.6cm}l(a_3\ldots a_k)l(b_1\ldots b_l))
+\tdun{$b_1+1$}\hspace{.5cm}\bullet l(a_1\ldots a_k)l(b_2\ldots b_l)))\\
&=m_{a_1 (a_2\ldots a_k)\bullet (b_1\ldots b_l)}+m_{(a_1-1)(a_2\ldots a_k\star b_1\ldots b_l)},
\end{align*}
by the preceding lemma. The result follows from an easy induction. \end{proof}\\

{\bf Remark.} In particular, $m_1\circ m_{b_1\ldots b_l}=0$.

\begin{cor}
Let $a_1\ldots a_k,b_1\ldots b_l$ be two words with letters in $\N^*$. Then $m_{a_1\ldots a_k}\bullet m_{b_1\ldots b_l}$
is a span of $m_w$, where $w$ is a word with $k+l$ letters and of weight $a_1+\ldots+a_k+b_1+\ldots+b_l$.
\end{cor}

Hence, $\B$ is a bigraded prelie algebra, with:
$$\B_{n,k}=Vect(m_{a_1\ldots a_k}\mid a_1+\ldots +a_k=n).$$
We put:
$$G=\sum_{k,n\geq 0} dim(\B_{n,k})X^nY^k.$$

\begin{prop} 
$\displaystyle G=\frac{XY}{1-X-X^2Y}=\sum_{k=1}^\infty \sum_{l=2k-1}^\infty \binom{l-k}{k-1}X^l Y^k$.
\end{prop}

\begin{proof} Note that $dim(\B_{n,k})$ is the number of words $a_1\ldots a_k$ of length $k$, such that $a_1,\ldots,a_{k-1}\geq 2$,
and $a_1+\ldots+a_k=n$. Hence:
$$G=\sum_{k=1}^\infty \left(\frac{X^2Y}{1-X}\right)^{k-1} \frac{XY}{1-X}
=\frac{XY}{1-X}\frac{1}{1-\frac{X^2Y}{1-X}}=\frac{XY}{1-X-X^2Y}.$$
An easy developement in formal series gives the second formula. \end{proof}

\subsection{An associative product on $\g_{\T(\N^*)}$}

We now define an associative product on $\g_{\T(\N^*)}$, in such a way that $\phi_{PL}$ becomes a morphism of Com-Prelie algebras.

\begin{prop}
We define a product $\shuffle$ on $\g_{\T(\N^*)}$ by:
$$B_p(s_1\ldots s_k) \shuffle B_q (t_1\ldots t_l)=\binom{p+q-k-l-2}{p-k-1} B_{p+q-1}(s_1\ldots s_k t_1\ldots t_l).$$
Then $\g_{\T(\N^*)}$ is a Com-Prelie algebra and $\phi_{PL}$ is a morphism of Com-Prelie algebras.
\end{prop}

\begin{proof}  As $\binom{p+q-k-l-2}{p-k-1}=\binom{p+q-k-l-2}{q-l-1}$, $\shuffle$ is commutative. 
Let $t=B_p( s_1\ldots s_k)$, $t'=B_q(\bullet t_1\ldots t_l)$ and $t''=B_r(u_1\ldots u_m)$. Then:
\begin{align*}
t\shuffle (t'\shuffle t'')=\underbrace{\binom{q+r-l-m-2}{q-l-1}\binom{p+q+r-k-l-m-3}{q+r-l-m-2}}_{A} B_{p+q+r-2}
(s_1\ldots s_kt_1\ldots t_lu_1\ldots u_m),\\
(t\shuffle t')\shuffle t''=\underbrace{\binom{p+q-k-l-2}{p-k-1}\binom{p+q+r-k-l-m-3}{p+q-k-l-2}}_{B} B_{p+q+r-2}
(s_1\ldots s_kt_1\ldots t_lu_1\ldots u_m).
\end{align*}
If $p\leq k$ or $q\leq l$ or $r\leq m$, then $A=B=0$. If $p>k$ and $q>l$ and $r>m$, then:
$$A=B=\frac{(p+q+r-k-l-m-3)!}{(p-k-1)!(q-l-1)!(r-m-1)!}.$$
So $\shuffle$ is associative. \\

Let $t_1=B_p( s_1\ldots s_k)$, $t_2=B_q(t_1\ldots t_l)$ and $t \in \T(\N^*)$. Then:
\begin{align*}
(t_1\shuffle t_2)\circ T&=\binom{p+q-k-l-2}{m-k-1}B_{p+q-1}(s_1\ldots s_kt_1\ldots t_lt)\\
&+\sum_{i=1}^k\binom{p+q-k-l-2}{p-k-1} B_{p+q-1}(s_1\ldots (s_i\bullet t) \ldots s_kt_1\ldots t_l)\\
&+\sum_{j=1}^l\binom{p+q-k-l-2}{p-k-1} B_{p+q-1}(s_1\ldots s_kt_1\ldots (t_j\bullet t) \ldots t_l),\\
(t_1\bullet t)\shuffle t_2&=\left(\sum_{i=1}^k B_p(s_1\ldots (s_i \bullet t)\ldots s_k)+B_p(s_1\ldots s_kt)\right)\shuffle t_2\\
&=\sum_{i=1}^k\binom{p+q-k-l-2}{p-k-1} B_{p+q-1}(s_1\ldots (s_i\bullet t) \ldots s_kt_1\ldots t_l)\\
&+\binom{p+q-k-l-3}{p-k-2}B_{p+q-1}(s_1\ldots s_kt_1\ldots t_lt),\\
t_1\shuffle (t_2\bullet t)&=t_1\shuffle \left(\sum_{j=1}^l B_q(t_1\ldots (t_j\bullet t)\ldots t_l)+B_q(t_1\ldots t_jt)\right)\\
&=\sum_{j=1}^l\binom{p+q-k-l-2}{p-k-1} B_{p+q-1}(s_1\ldots s_kt_1\ldots (t_j\bullet t) \ldots t_l)\\
&+\binom{p+q-k-l-3}{p-k-1}B_{p+q-1}(s_1\ldots s_kt_1\ldots t_lt).
\end{align*}
As $\displaystyle \binom{p+q-k-l-3}{p-k-2}+\binom{p+q-k-l-3}{p-k-1}=\binom{p+q-k-l-2}{p-k-1}$, 
we obtain $(t_1\shuffle t_2)\bullet t=(t_1\bullet t)\shuffle t_2+t_1\shuffle (t_2\bullet t)$. So $\g_{\T(\N^*)}$ is Com-Prelie.\\

Let $t_1=B_p(s_1\ldots s_k)$ and $t_2=B_q( t_1\ldots t_l)$. If $k\geq p$, then $\displaystyle \binom{p+q-k-l-2}{p-k-1}=0$,
so $t_1\shuffle t_2=0$. By proposition \ref{28}, $\phi_{PL}(t_1)=0$, so $\phi_{PL}(t_1\shuffle t_2)
=\phi_{PL}(t_1)\shuffle \phi_{PL}(t_2)=0$. Similarly, if $l\geq q$, $\phi_{PL}(t_1\shuffle t_2)=\phi_{PL}(t_1)
\shuffle \phi_{PL}(t_2)=0$. If $k<p$ and $l<q$, we put $w_i=\phi_{PL}(s_i)$ and $w'_j=\phi_{PL}(t_j)$. Then:
\begin{align*}
\phi_{PL}(t_1)\shuffle \phi_{PL}(t_2)&=x_0w_1\shuffle \ldots \shuffle x_0w_k \shuffle x_1^{p-1-k}
\shuffle x_0w'_1\shuffle \ldots \shuffle x_0w'_l \shuffle x_1^{q-1-l}\\
&=\binom{p+q-k-l-2}{p-k-1}x_0w_1\shuffle \ldots x_0w'_l\shuffle x_1^{p+q-k-l-2}\\
&=\binom{p+q-k-l-2}{p-k-1}\phi_{PL}(B_{p+q-1}(s_1\ldots s_k t_1\ldots t_l))\\
&=\phi_{PL}(t_1\shuffle t_2).
\end{align*}
So $\phi_{PL}$ is a Com-Prelie algebra morphism. \end{proof}\\

{\bf Remark.} By the proof of proposition \ref{28}, we have a commutative diagram of prelie algebra morphisms:
$$\xymatrix{\g_{\PT(\{1,2\}}\ar[r]^{\phi_{CPL}}&\B\\
\g_{\T(\N^*)}\ar[u]^{\psi} \ar[ru]_{\phi_{PL}}&}$$
Moreover, $\phi_{CPL}$ is a morphism of Com-Prelie algebra. With the commutative, associative product previously defined on $\g_{\T(\N^*)}$,
$\phi_{PL}$ is now a morphism of Com-Prelie algebra. However, $\psi$ is not compatible with $\shuffle$. Indeed,
$\psi(\tddeux{$2$}{$1$})=\psi(\tdun{$2$})\bullet \psi(\tdun{$1$})=\tddeux{$2$}{$1$}$, so:
$$\psi(\tddeux{$2$}{$1$})\shuffle \psi(\tddeux{$2$}{$1$})=\tddeux{$2$}{$1$}\shuffle \tddeux{$2$}{$1$}
=\hdquatretreize{$2$}{$1$}{$2$}{$1$}.$$
Moreover, $\tddeux{$2$}{$1$} \shuffle \tddeux{$2$}{$1$}=\tdtroisun{$3$}{$1$}{$1$}$, so:
$$\psi(\tddeux{$2$}{$1$} \shuffle \tddeux{$2$}{$1$})=\psi(\tdun{$3$})\bullet \psi(\tdun{$1$})\psi(\tdun{$1$})
=\frac{1}{2}\hddeux{$2$}{$2$} \bullet \tdun{$1$}\tdun{$1$}
=\hdquatretreize{$2$}{$1$}{$2$}{$1$}+\hdquatresept{$2$}{$1$}{$1$}{$2$}.$$

\section{Appendix}

\subsection{Enumeration of partitioned trees}

Let $d \geq 1$. For all $n \geq 1$, let $f_n$ be the number of partitioned trees decorated by $\{1,\ldots,d\}$ with $n$ vertices
and let $t_n$ be the number of partitioned trees decorated by $\{1,\ldots,d\}$ with $n$ vertices and one root. By convention, $f_0=1$. We put:
$$T=\sum_{n=1}^\infty t_n X^n, \: F=\sum_{n=0}^\infty f_n X^n.$$
Let $V_T$ be the vector space generated by the set of partitioned trees decorated by $\{1,,\ldots,d\}$
and $V_F$ be the vector space generated by the set of partitioned trees decorated by $\{1,,\ldots,d\}$ with only one root.
There is a bijection:
$$\begin{cases}
S(V_T)&\longrightarrow V_F\\
t_1\ldots t_k&\longrightarrow t_1\shuffle \ldots \shuffle t_k.
\end{cases}$$
Hence:
\begin{equation}
\label{E2}F=\prod_{i=1}^\infty \frac{1}{(1-X^k)^{t_k}}.
\end{equation}
There is a bijection:
$$\begin{cases}
\displaystyle \bigoplus_{i=1}^dS(V_F)&\longrightarrow V_T\\
(F_{1,1}\ldots,F_{1,k_1},\ldots,F_{d,1}\ldots F_{d,k_d})&\longrightarrow
\displaystyle \sum_{i=1}^d \tdun{$i$}\bullet(F_{i,1}\ldots F_{i,k_i}).
\end{cases}$$
This gives:
\begin{equation}
\label{E3}T=dX \prod_{i=1}^\infty \frac{1}{(1-X^k)^{f_{k-1}}}.
\end{equation}
Formulas (\ref{E2}) and (\ref{E3}) allow to compute inductively $f_k$ and $t_k$ for all $k\geq 1$. This gives:
$$\begin{cases}
f_1&=\displaystyle d\\
f_2&=\displaystyle \frac{d(3d+1)}{2}\\
f_3&=\displaystyle \frac{d(19d^2+9d+2)}{6}\\
f_4&=\displaystyle \frac{d(63d^2+34d^2+13d+2)}{8}\\
f_5&=\displaystyle \frac{d(644d^4+400d^3+175d^2+35d+6)}{30}
\end{cases}$$
Here are examples of $f_n$ for $d=1$ or $2$:
$$\begin{array}{|c|c|c|c|c|c|c|c|c|c|c|}
\hline n&1&2&3&4&5&6&7&8&9&10\\
\hline d=1&1&2&5&14&42&134&444&1518&5318&18989\\
\hline d=2&2&7&32&167&952&5759&36340&236498&1576156&10702333\\
\hline\end{array}$$
The row $d=1$ is sequence $A035052$ of \cite{Sloane}.

\subsection{Study of the dendriform structure on admissible words}

\label{s52} We here study the dendriform algebra $\K\langle\N^*\rangle$ of proposition \ref{32}. 
It is clearly commutative, via the bijection from $Sh_\prec(k,l)$ to $Sh_\succ(l,k)$ given by the composition (on the left) by the permutation
$(l+1\ldots l+k\: 1\ldots l)$: in other terms, it is a Zinbiel algebra \cite{Loday3}. 

Let $V$ be a vector space. The shuffle dendriform algebra $Sh(V)$ is $T_+(V)$, with the products given by:
\begin{eqnarray*}
(a_1\ldots a_k)\prec(b_1\ldots b_l)&=&\sum_{\zeta \in Sh_\prec(k,l)}\zeta.a_1\ldots a_k b_1\ldots b_{k-1}b_k\\
(a_1\ldots a_k)\succ(b_1\ldots b_l)&=&\sum_{\zeta \in Sh_\succ(k,l)}\zeta.a_1\ldots a_{k-1}a_k b_1\ldots b_k.
\end{eqnarray*}
Moreover, this is the free commutative dendriform algebra generated by $V$, that is to say if $A$ is a commutative dendriform algebra and
$f:V\longrightarrow A$ is any linear map, there exists a morphism of dendriform algebras $\phi:Sh(V)\longrightarrow A$ such that $\phi_\mid V=f$.
As $a_1\ldots a_k \succ b=a_1\ldots a_k b$ in $Sh(V)$ for all $a_1,\ldots,a_k,b \in V$, this morphism $\phi$ is defined by:
$$\phi(a_1\ldots a_k)=(\ldots(a_1\succ a_2)\succ a_3)\ldots)\succ a_k.$$

\begin{prop}\begin{enumerate}
\item Let $V$ be the space generated by the words $1^k i$, $k\in \N$, $i\geq 1$. Then  $K \langle\N^*\rangle$ is isomorphic, as a dendriform algebra, 
to $Sh(V)$.
\item Let $A$ be the subspace of  $K \langle\N^*\rangle$ generated by admissible words. Then it is a dendriform subalgebra of  $K \langle\N^*\rangle$.
Moreover, if $W$ is the space generated by the letters $i$, $i\geq 1$, then $A$ is isomorphic, as a dendriform algebra, to $Sh(W)$.
\end{enumerate}\end{prop}

\begin{proof}
Let $w=a_1\ldots a_k$ be a word with letters in $\N^*$. We denote by $o(w)$ the sequence of indices $j\in \{1,\ldots,k-1\}$ such that $a_j\neq 1$.
This sequences are totally ordered in this way: $(j_1,\ldots,j_k)<(j'_1,\ldots,j'_l)$ if there exists a $p$ such that $j_k=j'_l$, $j_{k-1}=j'_{l-1}$,
$\ldots$, $j_{k-p+1}=j'_{l-p+1}$, $j_{k-p}<j'_{l-p}$, with the convention $j_0=j_{-1}=\ldots=j'_0=j'_{-1}=\ldots=0$. \\

Let $\phi:Sh(V)\longrightarrow K \langle\N^*\rangle$ be the unique morphism of dendriform algebras which extends the identity of $V$. Then:
\begin{eqnarray*}
\phi((1^{k_1-1}a_1)\ldots (1^{k_n-1}a_n))&=&=1^{k_1-1}(a_1+1)\ldots 1^{k_{n-1}-1}(a_{n-1}+1)1^{k_n-1} a_n\\
&&+\mbox{words $w'$ such that $o(w')>(k_1,\ldots,k_{n-1})$}.
\end{eqnarray*}
By triangularity, $\phi$ is an isomorphism. Moreover, for all $a_1,\ldots,a_n \geq 1$:
$$\phi(a_1\ldots a_n)=(a_1+1)\ldots(a_{n-1}+1)a_n.$$
Consequently, $\phi(Sh(W))=A$, so $A$ is a dendriform subalgebra of $K\langle\N^*\rangle$ and is isomorphic to $Sh(W)$. \end{proof}

\subsection{Freeness of the pre-Lie algebra $\g_{\PT(\D)}$}

{\bf Notations.} Let $k\geq 1$, $d_1,\ldots,d_k \in \D$ and let $F_1,\ldots, F_k$ be decorated partitioned forests. We put:
$$B_{d_1,\ldots,d_k}(F_1,\ldots F_k)=(\tdun{$d_1$}\bullet F_1) \shuffle\ldots \shuffle (\tdun{$d_k$}\bullet F_k).$$
Note that any partitioned tree can be written under the form $B_{d_1,\ldots,d_k}(F_1,\ldots F_k)$. This writing is unique up to a common permutation
of the $d_i$'s and the $F_i$'s. 

\begin{prop}
We define a coproduct $\delta$ on $\g_{\PT(\D)}$ in the following way: 
for any decorated partitioned tree $t=B_{d_1,\ldots,d_k}(t_{1,1}\ldots t_{1,n_1},\ldots, t_{k,1}\ldots t_{k,n_k})$,
$$\delta(t)=\frac{1}{k} \sum_{i=1}^k  \sum_{j=1}^{n_i}B_{d_1,\ldots,d_k}(t_{1,1}\ldots t_{1,n_1},
\ldots,t_{i,1}\ldots t_{i,j-1}t_{i,j+1}\ldots t_{i,n_i},\ldots,t_{k,1}\ldots t_{k,n_k})\otimes t_{i,j}.$$
\begin{enumerate}
\item For all $x \in \g_{\PT(\D)}$, $(\delta \otimes Id)\circ \delta(x)=(23)(\delta \otimes Id)\circ \delta(x)$.
\item For all $x,y\in \g_{\PT(\D)}$, $\delta(x \bullet y)=x \otimes y+\delta(x) \bullet y$.
\end{enumerate}\end{prop}

\begin{proof} 1. Let $t=B_{d_1,\ldots,d_k}(t_{1,1}\ldots t_{1,n_1},\ldots, t_{k,1}\ldots t_{k,n_k})$.
For all $i,j$, we put:
$$t/t_{i,j}= B_{d_1,\ldots,d_k}(t_{1,1}\ldots t_{1,n_1},
\ldots,t_{i,1}\ldots t_{i,j-1}t_{i,j+1}\ldots t_{i,n_i},\ldots,t_{k,1}\ldots t_{k,n_k}).$$
Then:
$$\delta(t)=\frac{1}{k}\sum_{i,j} t/t_{i,j}\otimes t_{i,j}.$$
Hence:
$$(\delta \otimes Id)\circ \delta(t)=\sum_{(i,j)\neq (i',j')} (t/t_{i,j})/t_{i',j'}\otimes t_{i',j'} \otimes t_{i,j}$$
As $(t/t_{i,j})/t_{i',j'}$ and $(t/t_{i',j'})/t_{i,j}$ are both the partitioned tree obtained by cutting $t_{i,j}$ and $t_{i',j'}$ in $t$, 
they are equal, so $(\delta \otimes Id)\circ \delta(t)$ is invariant under the action of $(23)$. \\

2. Let $t'$ be a decorated partitioned tree. 
\begin{eqnarray*}
\delta(t\bullet t')&=&\sum_{i=1}^k \delta(B_{d_1,\ldots,d_k}(t_{1,1}\ldots t_{1,n_1},\ldots,
t_{i,1}\ldots t_{i,n_i}t',\ldots,t_{k,1}\ldots t_{k,n_k}))\\
&&+\sum_{i,j} \delta(B_{d_1,\ldots,d_k}(t_{1,1}\ldots t_{1,n_1},\ldots,
t_{i,1}\ldots t_{i,j}\bullet t' \ldots t_{i,n_i},\ldots,t_{k,1}\ldots t_{k,n_k}))\\
&=&\frac{1}{k} kt\otimes t'+\frac{1}{k}\sum_i\sum_{i',j'} B_{d_1,\ldots,d_k}(t_{1,1}\ldots t_{1,n_1},\ldots,
t_{i,1}\ldots t_{i,n_i}t',\ldots,t_{k,1}\ldots t_{k,n_k})/t_{i',j'}\otimes t_{i',j'} \\
&&+\frac{1}{k}\sum_{(i,j) \neq (i',j')}B_{d_1,\ldots,d_k}(t_{1,1}\ldots t_{1,n_1},\ldots,
t_{i,1}\ldots t_{i,j}\bullet t' \ldots t_{i,n_i},\ldots,t_{k,1}\ldots t_{k,n_k})/t_{i',j'}\otimes t_{i',j'}\\
&&+\frac{1}{k}\sum_{i,j}t/t_{i,k}\otimes  t_{i,j} \bullet t' \\
&=&t\otimes t'+\sum t^{(1)} \otimes t^{(2)} \bullet t'+\sum t^{(1)}\otimes t^{(2)}\bullet t'.
\end{eqnarray*}
So $\delta(t \bullet t')=t\otimes t'+\delta(t) \bullet t'$. \end{proof} \\

By Livernet's pre-Lie rigidity theorem \cite{Livernet}:

\begin{cor}
The pre-Lie algebra $\g_{\PT(\D)}$ is freely generated by $Ker(\delta)$.
\end{cor}

{\bf Remarks.} \begin{enumerate}
\item  It is not difficult to prove that for any $x,y \in \g_{\PT(\D)}$:
$$\delta(x \shuffle y)=\sum x^{(1)}\otimes x^{(2)}\shuffle y+\sum y^{(1)} \otimes x \shuffle y^{(2)}.$$
Hence, $Ker(\delta)$ is an algebra for the product $\shuffle$.
\item Here are elements of $Ker(\delta)$ in the non decorated case. Let $t_1,t_2,t_3,t_4$ be partitioned trees.
\begin{eqnarray*}
X&=&B(t_1t_2,1)-B(t_1,t_2),\\
Y&=&B(t_1t_2t_3,1,1)-B(t_1t_2,t_3,1)-B(t_1t_3,t_2,1)-B(t_2t_3,t_1,1)+2B(t_1,t_2,t_3),\\
Z&=&B(t_1t_2t_3t_4,1)-B(t_1t_2t_3,t_4)-B(t_1t_2t_4,t_3)-B(t_1t_3t_4,t_2)-B(t_2t_3t_4,t_1)\\
&&+B(t_1t_2,t_3t_4)+B(t_1t_3,t_2t_4)+b(t_1t_4,t_2t_3),\\
T&=&B(t_1t_2,t_3t_4,1,1)+B(t_1t_3,t_2t_4,1,1)+B(t_1t_4,t_2t_3,1,1)-B(t_1t_2,t_3,t_4,1)\\
&&-B(t_1t_3,t_2,t_4,1)-B(t_1t_4,t_2,t_3,1)-B(t_2t_3,t_1,t_4,1)-B(t_2t_4,t_1,t_3,1)\\
&&-B(t_3t_4,t_1,t_2,1)+3B(t_1,t_2,t_3,t_4).
\end{eqnarray*}

\end{enumerate}

\bibliographystyle{amsplain}
\bibliography{biblio}

\end{document}